\documentclass[10pt,a4paper,reqno]{amsart} 
\usepackage{amsmath}
\usepackage{bbm}
\usepackage{mathtools}
\usepackage{amssymb}
\usepackage{graphicx}
\usepackage{color}
\usepackage{latexsym}
\usepackage{comment}
\usepackage[english]{babel} 
\usepackage{stmaryrd}
\usepackage{amssymb} 
\usepackage[mathcal]{eucal}
\usepackage{fancybox}
\usepackage{pifont, xcolor, tikz}
\usepackage{palatino}

\newcommand{\pd}[2]{\frac{\partial#1}{\partial#2}}

\newcommand{\fig}[3]{
\begin{figure}[ht!]
\begin{center}
 \includegraphics #1
 \end{center}
\vspace{-7pt}
\caption{ #2}
\label{#3}
\end{figure}
}

\newcommand{\ens}[1]{\left\{#1\right\}}

\newtheorem{prop}{Proposition}

\newtheorem{lem}[prop]{Lemma}
\newtheorem{defi}[prop]{Definition}

\newtheorem{theo}[prop]{Theorem}

\newtheorem{cor}[prop]{Corollary}

\usetikzlibrary{positioning}
\usetikzlibrary{decorations.pathreplacing}

\tikzset{partition/.style={fill,circle,inner sep=1pt},
         part/.style={baseline=0,scale=0.5,bend left=45},
         partlabel/.style={below}}
\usetikzlibrary{shapes,arrows}
\tikzstyle{pnt}=[draw,ellipse,fill,inner sep=1pt]

\catcode`\@=11
\def\section{\@startsection{section}{1}%
 \z@{.7\linespacing\@plus\linespacing}{.5\linespacing}%
 {\normalfont\bfseries\scshape\centering}}

\def\subsection{\@startsection{subsection}{2}%
  \z@{.5\linespacing\@plus\linespacing}{.5\linespacing}%
  {\normalfont\bfseries\scshape}}

\def\subsubsection{\@startsection{subsubsection}{3}%
 \z@{.5\linespacing\@plus\linespacing}{-.5em}%{.5\linespacing}%
  {\normalfont\bfseries\itshape}}
\catcode`\@=12

\title{Terminal chords in connected chord diagrams} 
\author{Julien Courtiel and Karen Yeats}
\thanks{JC is supported by a PIMS postdoctoral fellowship; KY is supported by an NSERC discovery grant. JC is grateful to Andrea Sportiello, Valentin Bonzom and Olivier Bodini for interesting discussions.}

\begin{document}
\maketitle

	\begin{abstract}		
	Rooted connected chord diagrams
%	, or equivalently connected matchings of $\{1,\dots, 2n\}$, 
	form a nice class of combinatorial objects.  Recently they were shown to index solutions to certain Dyson-Schwinger equations in quantum field theory.  Key to this indexing role are certain special chords which are called terminal chords.  Terminal chords provide a number of combinatorially interesting parameters on rooted connected chord diagrams which have not been studied previously.  Understanding these parameters better has implications for quantum field theory.		
			
	Specifically, we show that the distributions of the number of terminal chords and the number of adjacent terminal chords are asymptotically Gaussian with logarithmic means, and we prove that the average index of the first terminal chord is $2n/3$.  Furthermore, we obtain a method to determine any next-to${}^i$ leading log expansion of the solution to these Dyson-Schwinger equations, and have asymptotic information about the coefficients of the log expansions.  		
	\end{abstract}

\section{Introduction}
In this paper we are interested in looking at the asymptotic behaviour of some rich and interesting, but somewhat unusual parameters on the combinatorial class of rooted connected chord diagrams.  Specifically, we are interested in certain chords known as \emph{terminal chords} which form the base case for a recursive decomposition of rooted connected chord diagrams and the indices of the terminal chords in a recursive ordering of the chords.  The reason for investigating these parameters is that they arose in \cite{MYchord} in series solutions to certain Dyson-Schwinger equations in quantum field theory.  In order to derive meaningful physics from these series solutions we need to better understand the asymptotics of these parameters.  The present paper is a first step towards this understanding.  Furthermore the combinatorics of these objects is interesting in its own right and these particular parameters are largely uninvestigated so far.

\subsection{Combinatorial setting}

Before explaining the physics context, let us set up what we need for chord diagrams. 

\begin{defi} A \emph{perfect matching} of a finite set $S$ is a set of pairs of $S$ such that every element of $S$ is in exactly one pair.  A \emph{chord diagram} with $n$ chords is a perfect matching of $\{1,2,\ldots, 2n\}$.   The \emph{root chord} of a chord diagram is the pair including $1$.
\end{defi}

As implied by the name, it is convenient to represent chord diagrams with dots and chords. Two conventions coexist in the literature: the circular one and the linear one. They respectively consist in drawing points $1,2,\ldots,2n$ on a circle in counterclockwise order, or on a line from left to right, and joining by a chord every two elements belonging to the same pair. The matching $\ens{\{1,4\},\{2,6\},\{3,5\}}$ has been drawn in these two ways in Figure \ref{fig:convention}.  The circular convention has been used in the previous papers, like \cite{MYchord}, but we are going to adopt here the linear convention for the rest of the document.

\fig{{C1}\quad \quad \quad \begin{tikzpicture}[scale=0.7]
\foreach \x in {1,2,...,6}{
\node[pnt,label=below:{$\x$}](\x) at (\x,0){};}
\draw(1)  to [bend left=45] (4);
\draw(3)  to [bend left=45] (5);
\draw(2)  to [bend left=45] (6);
\end{tikzpicture}}{\textit{Left:} circular convention. \textit{Right:} linear convention.}{fig:convention}

\begin{defi}
 The \emph{oriented intersection graph} of a  chord diagram $C$ is the digraph with a vertex for each chord of $C$ and an oriented edge from chord $\{a,b\}$ to chord $\{c,d\}$ whenever $a<c<b<d$.
  A chord diagram is \emph{connected} if its oriented intersection graph is connected.
   A chord is \emph{terminal} if its vertex in the oriented intersection graph has no outgoing edges.
\end{defi}
For instance, the oriented intersection graph of the chord diagram of Figure~\ref{fig:convention} is the tree where $\ens{1,4}$ is the root vertex, and $\ens{2,6},\ens{3,5}$ its two children.
This chord diagram is connected, and the terminal chords are $\{2,6\}$ and $\{3,5\}$.

The chords inherit an order by the smaller of their endpoints. 
% We call this the \emph{counterclockwise order} but 
This is not the order that we want to be working with.

\begin{defi}
The \emph{intersection order} of the chords of a rooted connected chord diagram $C$ is defined as follows.
\begin{itemize}
  \item The root chord of $C$ is the first chord in the intersection order.
  \item Remove the root chord of $C$ and let $C_1, C_2, \ldots, C_n$ be the connected components of the result ordered by their first vertex.
  \item For the intersection order of $C$, after the root chord come all the chords of $C_1$ ordered inductively in the intersection order, then all the chords of $C_2$ ordered by intersection order, and so on.
\end{itemize}
\end{defi}

The chord diagram of Figure \ref{fig:intersection} is an example of a chord diagram where the intersection order is different from the order by the smaller of their endpoints.

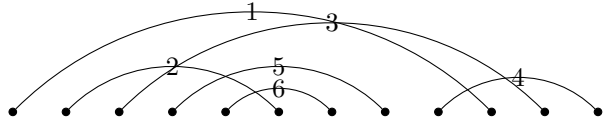
\begin{figure}[ht!]
\begin{center}
\begin{tikzpicture}[scale=0.7]
\foreach \x in {1,2,...,12}{
\node[pnt](\x) at (\x,0){};}
\draw(1)  to [bend left=45]  node [midway] {$1$}  (10);
\draw(3)  to [bend left=45] node [midway] {$3$} (11);
\draw(5)  to [bend left=45] node [midway] {$6$} (7);
\draw(2)  to [bend left=45] node [midway] {$2$} (6);
\draw(9)  to [bend left=45] node [midway] {$4$} (12);
\draw(4)  to [bend left=45] node [midway] {$5$} (8);
\end{tikzpicture}
% \quad \quad \quad \begin{tikzpicture}[scale=0.7]
%\foreach \x in {1,2,...,8}{
%\node[pnt](\x) at (\x,0){};}

%\draw(1)  to [bend left=45] (4);
%\draw(3)  to [bend left=45] (6);
%\draw(5)  to [bend left=45] (7);
%\draw(2)  to [bend left=45] (8);
%\end{tikzpicture}
 \end{center}
\vspace{-7pt}
\caption{Example of a connected chord diagram and its intersection order.}
\label{fig:intersection}
\end{figure}

Our primary interest is in the terminal chords and their indices in intersection order.  We are interested in questions such as
\begin{itemize}
  \item How many terminal chords does a chord diagram have?
  \item What is the index of the first terminal chord?
  \item How many pairs of terminal chords are adjacent in the intersection order?
  \item What can we say about the gaps between indices of successive terminal chords in intersection order?
\end{itemize}
Now we are going to explain why these questions are relevant from a physical point of view.

\subsection{Physical background}

Dyson-Schwinger equations are an important class of equations in quantum field theory.  They are the quantum analogues of the classical equations of motion.  They are usually written as integral equations and their recursive structure mirrors the decomposition of Feynman graphs into subgraphs.  

Because of this recursive structure there is another, more combinatorial way to think about them.  Namely, they are functional equations for a sort of weighted generating function.  More specifically, the Green functions of a quantum field theory can be thought of as the sum over all Feynman graphs of the theory satisfying certain properties (for example graphs which are 1 particle irreducible -- that is 2-edge-connected -- and have a fixed set of external edges) and weighted by their Feynman integrals.  Thus the Green functions are weighted generating functions of Feynman graphs with highly nontrivial weights.  The Green functions are solutions to the Dyson-Schwinger equations, or, viewed the other way around, the Dyson-Schwinger equations are certain functional equations for these weighted generating functions.

\medskip

In \cite{MYchord} one of the authors along with Nicolas Marie looked at one particular family of Dyson-Schwinger equations given below in \eqref{DSE}.  This family of Dyson-Schwinger equations corresponds to the physical situation where we consider all graphs made by inserting a fixed one loop propagator graph into itself in one insertion place.  Combinatorially this means that the graphs we are interested in are in bijection with plane rooted trees.
% given by the specification
%\[
%  \mathcal{G} = \mathcal{Z}\times \textsc{Seq}(\mathcal{G})
%\]   
%***Julien, please rephrase this however you like your specifications best.*** 
For instance, inserting the graph
\[
\includegraphics{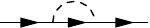}
\]
into itself in all possible ways gives a class of graphs which fits into this situation.  One example from this class is 
\[
\includegraphics{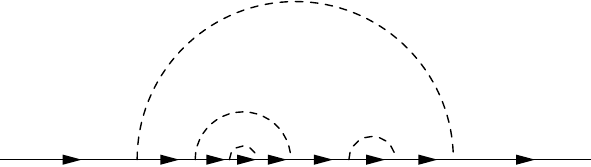}
\]
which corresponds to the rooted tree
\[
\includegraphics{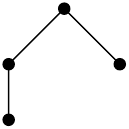}.
\]

The Dyson-Schwinger equations considered in \cite{MYchord} are those which can be written in the following form
\begin{equation}\label{DSE}
 G(x,L) = 1-xG\left(x,\frac{\partial}{\partial (-\rho)}\right)^{-1} (e^{-L\rho} -1)F(\rho) |_{\rho=0}
\end{equation}
where $F(\rho)$ is the Laurent expansion of a regularized Feynman integral for the one loop graph which generated the graph class in question.  Note that $G(x,\partial/\partial (-\rho))^{-1}$ is acting as a differential operator on $(e^{-L\rho} -1)F(\rho)$.  Analytically there are subtleties since such an operator is only a pseudo-differential operator.  However, we are concerned here solely with series and so interpreting \eqref{DSE} as an equation in formal series everything is well-defined.

For the specific example graphs given above, viewed in Yukawa theory, the equation was solved by Broadhurst and Kreimer in \cite{bkerfc}.  In this case $F(\rho) = 1/(\rho(1-\rho))$.  They didn't write the Dyson-Schwinger equation in the form of \eqref{DSE} but rather in a more usual physical form as an integral equation (see Examples 3.5 and 3.7 of \cite{Ymem} for how to convert to the form above).  However, for the present purposes, we can simply take \eqref{DSE} as the starting point.

\medskip

The main result of \cite{MYchord} is a series solution to \eqref{DSE} indexed by rooted connected chord diagrams.  

\begin{theo}[Theorem 4.13 of \cite{MYchord}]
Suppose $F(\rho) = \sum_{i\geq 0} f_i \rho^{i-1}$.  Given a rooted connected chord diagram $C$ let the indices of its terminal chords in intersection order be $b(C) = t_1 < t_2 < \cdots < t_k$.
Then
\begin{equation}\label{sol}
G(x,L) = 1 - \sum_{i \geq 1}\frac{(-L)^i}{i!}\sum_{\substack{C\\b(C) \geq i}}x^{|C|}f_{b(C)-i}f_0^{|C|-k}\prod_{j = 2}^{k} f_{t_j-t_{j-1}}
\end{equation}
solves \eqref{DSE}, where the sum is over rooted connected chord diagrams with the indicated restriction.
\label{theo:MYchord}
\end{theo}
Note that $k$ and the $t_j$ depend on $C$ in \eqref{sol}, but this has been left implicit to keep the notation from getting too heavy.
More general Dyson-Schwinger equations have similar chord diagram expansions (see \cite{HYchord}).

To better understand $G(x,L)$ we see from \eqref{sol} that the keys are to understand the number of terminal chords, the index of the first terminal chord, and the differences between indices of successive terminal chords.

\subsection{Structure of the document}

 In Section \ref{sec:eb}, we set up the enumerative background of the present article. We introduce what is known on exact and asymptotic enumeration of connected chord diagrams. We also give some results which we use in this article: first we cite two theorems from the theory of analytic combinatorics; then we establish an asymptotic expansion of the ratio $c_{n}/c_{n-1}$, where $c_n$ is the number of connected chord diagrams.

In Section \ref{sec:llterms}, we study the leading-log expansion and next-to${}^i$-leading log expansions of the solution of~\eqref{DSE}.  These are a way of organizing the double expansion of $G(x,L)$ and are relevant in quantum field theory. To do so, we are led to enumerate the connected chord diagrams $C$ such that the first terminal chord is close to the last chord. We give recurrences that characterize these numbers and their exponential generating functions. We then establish asymptotic estimates on them. There will be implications on the physical level: we show that the dominant terms in all the log expansions only involve $f_0$ and $f_1$, respectively the residue and the constant term of $F(\rho)$.

In Section \ref{sec:stat}, we study parameters on uniform large connected chord diagrams. We state a generic theorem that shows that numerous laws on connected chord diagrams obey to a Gaussian law, with a mean of the form $\lambda n + \mu \ln n$ and a logarithmic variance. In particular, we prove that the average number of terminal chords is $\ln n$. Finally, we see that the average index of the first terminal chord, a parameter which does not satisfy the above-mentioned theorem, is $2n/3$. The methods used are new and interesting as combinatorics.

\section{Enumerative background} \label{sec:eb}

\subsection{A brief historical background on connected chord diagrams}

Chord diagrams and their enumeration are not only relevant in quantum field theory; they also appear in various other areas of mathematics: knot theory \cite{Stoimenow,bori2000,zagier2001} (in particular the Vassiliev invariants), graph sampling \cite{acan2013}, analysis of computer structures \cite{flajolet-datastructure}, and even bioinformatics \cite{combRNA, andersen}.
% (chord diagrams encode the RNA secondary structure). 
% In this context, the enumeration of these objects has been the subject of numerous studies.

Concerning more particularly the \emph{connected} chord diagrams, Touchard seems to be the first person in 1952 to be interested in their enumeration \cite{touchard}. More precisely, he characterized the number of connected diagrams with $n$ chords and $m$ crossings as a solution of a system of equations. 
Subsequently, Stein provided an explicit recurrence relation for the number of connected chord diagrams (but without considering this time the number of crossings) \cite{stein}, as stated in the following proposition.

\begin{prop}[Stein \cite{stein}]
Let $c_n$ be the number of connected diagrams with $n$ chords.  These numbers satisfy the relations $c_1 = 1$ and for $n \geq 2$,
\begin{equation}
c_n = (n-1)  \, \sum_{k=1}^{n-1} c_k\,c_{n-k}.
\label{eq:rec}
\end{equation}
\end{prop}

This formula has also been shown by Nijenhuis and Wilf, but this time thanks to a constructive combinatorial proof \cite{nijenhuis-wilf}. Let us also mention than \eqref{eq:rec} is equivalent to
\begin{equation}
c_n =  \sum_{k=1}^{n-1} (2\,k-1)\, c_k\,c_{n-k}.
\label{eq:rec2}
\end{equation}

As for the asymptotic behaviour of the number $c_n$ of connected diagrams with $n$ chords, Stein and Everett gave the estimate
\[
c_n \sim \frac 1 e (2\,n-1)!!
\]
 in \cite{stein-everett}. In particular, since the number of (non necessarily connected) diagrams with $n$ chords is $(2n-1)!!$, this implies that a large random chord diagram is connected with a probability $e^{-1}$. 
 
Some decades later, Flajolet and Noy  refined this result. Indeed, they proved in \cite{flajolet-noy} that the number of connected components in a large random chord diagram (minus $1$) follows a Poisson law of parameter $1$. Moreover, they showed that if $L_n$ denotes the size of the largest component in a random diagram with $n$ chords, then $n - L_n$ is also distributed like a Poisson law of parameter $1$. 

Recently, Michael Borinsky computed in~\cite{michi} an asymptotic expansion of the number of connected diagrams $c_n$, along with cardinals of similar objects.

\subsection{Preliminaries on analytic combinatorics}

The majority of our proofs are based on the reference book by Flajolet and Sedgewick \cite{flajolet-sedgewick}. We present here the two main analytic combinatorics theorems of this paper.

First of all, let us mention that we use in this document two different notions of generating function. Given a sequence $a_n$, the \textit{ordinary generating function} of the numbers $a_n$ is defined as $\sum_{n \geq 0} a_n z^n$, while the \textit{exponential generating function} is defined as $\sum_{n \geq 0} a_n z^n/n!$. Both notions have their advantages and drawbacks, especially when we try to enumerate chord diagrams. That is why we will juggle the two notions.

The first theorem, maybe the most representative of the theory, is called the \textit{transfer theorem}. It relates the singular expansion of the series and the asymptotic behaviour of its coefficients.

\begin{theo} [Transfer Theorem]  \label{theotransfer}
Consider $\Delta$ a complex domain of the form
$$ \left\{ z \ \  \left| \ \  |z| < \rho + \varepsilon \right. \right\} \ \bigcap \  \left\{ z\ \  \left| \ \   {|\textrm{Arg}(z-\rho)|} > \alpha   \right. \right\},$$
with $\varepsilon > 0$ and $\alpha \in (0,\pi/2)$.
%(see Figure \ref{deltadomain}).
%\fig{[scale = 1.2]{deltadomain}}{A typical domain for Theorem~\ref{theotransfer}.}{deltadomain}

Let $f(z) = \sum_{n \geq 0} f_n z^n$ be an analytic function on $\Delta$. If the singular behaviour of $f$ in the vicinity of $\rho$ is 
\[f(z) \underset{z \rightarrow \rho}\sim c \, (1-z/\rho)^{-\alpha} \ln(1-z/\rho)^{\beta},\]
where $\alpha$ is a complex number which does not belong to $\mathbb Z_{<0}$, $\beta$ any integer and $c$ a non zero constant, then
\[f_n \sim \frac{c}{\Gamma(\alpha)} \rho^{-n} n^{\alpha - 1} (\ln n)^\beta.\]
\end{theo}

In this article, the analyticity of our functions on a domain with the same shape as $\Delta$ is generally obvious (mainly because we have explicit expressions). The justification of analyticity will be then omitted, except if there is a subtlety to stress.

\noindent \textbf{Example of use of transfer theorem.} Consider the exponential generating function of $(2n-3)!!$. We will prove that it is equal to $1- \sqrt{1-z}$. By the transfer theorem, the numbers $(2n-3)!!/n!$ are equivalent to $-n^{-3/2}/\Gamma(-1/2)= n^{-3/2}/(2\sqrt \pi)$. This can be checked by the Sterling formula.

The next theorem deals with the Quasi-Powers Theorem, stated under a form which will be useful for us.

\begin{theo}[Theorem IX.11 of \cite{flajolet-sedgewick} -- Quasi-Powers Theorem]  
\label{theo:qp}
Let $F(z,u) = \sum_{n \geq n_0,k\geq 0} f_{n,k} z^n u^k$ be a bivariate function with non-negative coefficients such that $\sum_{k \geq 0} f_{n,k} = 1$  for $n \geq n_0$.
\begin{itemize}
\item[(i)] \emph{Analytic representation.} The function $F(z,u)$ admits the representation 
\[F(z,u) = A(z,u) + B(z,u)(1-z)^{-\alpha(u)}\]
where $A(z,u)$ and $B(z,u)$ are analytic  on a domain of the form $\{ (z,u) \ | \ |z| \leq r,\ |u-1| < \varepsilon\}$, with $r > 1$ and $\varepsilon > 0$. Assume also that $\alpha(u)$ is analytic at $1$ such that $\alpha(1)$ is not a non-positive integer and $B(1,1) \neq 0$.
\item[(ii)] \emph{Variability condition.} One has $\alpha'(1)+\alpha''(1) \neq 0$.
\end{itemize}
Then the random variable $X_n$ such that $\mathbb P(X_n = k) = f_{n,k}$ converges in distribution  to a Gaussian variable. The corresponding mean is $\alpha'(1)\ln n$ and the variance is $(\alpha'(1)+\alpha''(1)) \ln n$.
\end{theo}

\subsection{A refined asymptotic result}

We will need a precise asymptotic expansion of $c_{n-1}/c_{n}$. The dominant term has already been given by Stein and Everett: they showed that $c_{n-1}/c_{n} \, \sim \, 1/2n$. This expansion can also be deduced from the article of Michael Borinsky~\cite{michi}.

\begin{prop}The ratio between the numbers of connected chord diagrams with $n-1$ arcs and $n$ arcs is asymptotically equivalent to
\[\frac {c_{n-1}} {c_{n}} = \frac 1 {2\,n} + \frac  1 {4\,n^2} -  \frac  1 {2\,n^3}  -  \frac  {29} {8\,n^4} +  O  \left( \frac 1 {n^5} \right).\]
\label{prop:expansion}
\end{prop}

The proof simply relies on what is often called \textit{bootstrapping}. We need first to establish a lemma which bounds the contribution of the central terms in the sum $(n-1)\,\sum_{k=1}^{n-1}\,\frac{c_k\,c_{n-k}}{c_n}$.

\begin{lem} If $c_n$ denotes the number of connected diagrams with $n$ chords, we have for fixed $j \geq 2$ the estimate
\begin{equation*}
(n-1)\,\sum_{k=j}^{n-j}\,\frac{c_k\,c_{n-k}}{c_n} = O\left(\frac 1 {n^{j-1}}\right).
\label{eq:bound}
\end{equation*}
\end{lem}

\begin{proof}  
 The core of the proof lies in the inequality
\begin{equation}
(2\,n-1)\,c_{n-1} < c_{n} < 2\,n\, c_{n-1},
\label{eq:boundratio}
\end{equation}
holding for all $n \geq 5$. This has been stated by Stein and Everett in \cite[Lemmas 3.1 and 3.4]{stein-everett} and was proved by a (technical) induction. 
Notice that this inequality justifies the previous estimate $c_{n-1}/c_{n} \, \sim \, 1/2n$.
For $n \geq 8$ and $k \in \ens{5,\dots,\lfloor\frac n 2\rfloor}$, we then have 
$$c_{k}\,c_{n-k} \leq \frac {2\,k}{2\,n-2\,k+1} \, c_{k-1}\,c_{n-k+1} \leq c_{k-1}\,c_{n-k+1}.$$
The inequality $c_{k}\,c_{n-k}  \leq c_{k-1}\,c_{n-k+1}$ is also true for $k \in \ens{2,3,4}$ since $c_{k}/c_{k-1} \leq c_4/c_3  < 7$ and thanks to \eqref{eq:boundratio}, we have $c_{n-k+1}/c_{n-k} \geq (2n-2k+1) \geq 9$ for every $n \geq 8$ and $k \in \ens{2,3,4}$ (Inequality \eqref{eq:boundratio} applies because $n-k+1 \geq 5$).
Therefore, we show by a basic induction that for every $n \geq 8$ and for $ 2 \leq  k \leq j \leq \lfloor\frac n 2\rfloor,$
\begin{equation} c_{k}\,c_{n-k} \leq c_{k-1}\,c_{n-k+1} \leq \dots \leq c_{j-1}\,c_{n-j+1} \leq c_{j}\,c_{n-j}. 
\label{eq:gros}
\end{equation}
Taking in particular the inequality $c_{k}\,c_{n-k} \leq c_{j-1}\,c_{n-j+1}$ and summing over $k \in \{j+1,n-j+1\}$, we obtain \[\sum_{k=j+1}^{n-j-1} c_{k}\,c_{n-k} \leq  n \,c_{j+1}\,c_{n-j-1}.\] Consequently,
\[(n-1)\,\sum_{k=j}^{n-j}\,\frac{c_k\,c_{n-k}}{c_n} = O\left(n\,\frac{c_{n-j}}{c_n} +  n^2\,\frac{c_{n-j-1}}{c_n}\right).\] 
But $c_{n-j}/c_n$ can be written as $\prod_{i=1}^{j} c_{n-i}/c_{n-i+1}$, so is equivalent to $1/(2n)^j$ (we have used the estimate $c_{n-1}/c_{n} \, \sim \, 1/2n$). Plugging this in the previous equality directly gives the lemma.
\end{proof}

Now let us describe how to find an expansion of $c_{n-1}/c_n$.

\begin{proof}[Proof of Proposition \ref{prop:expansion}]
 By combining the previous lemma and \eqref{eq:rec}, we deduce that 
\[ 1 = 2(n-1)\sum_{k=1}^j \frac{c_k\,c_{n-k}}{c_n} + O\left(\frac 1 {n^j}\right)\]
holds for $j \geq 1$. For $j=1$, we recover $c_{n-1}/c_{n} \, \sim \, 1/2n$. For $j=2$, we have
\[ 1 = 2(n-1)\frac{c_{n-1}}{c_n} + 2(n-1) \frac{c_{n-2}}{c_{n-1}} \frac{c_{n-1}}{c_n}  + O\left(\frac 1 {n^2}\right).\] 
Setting $ c_{n-1}/c_{n} = 1/2n + d_n$ leads to
\[0= - \frac 1 n + 2\,(n-1)\,d_n  + 2\,(n-1)\,\left( \frac1 {2(n-1)}+o\left(\frac 1 n\right) \right) \left( \frac 1 {2\,n}+o\left(\frac 1 n\right) \right)  + o\left(\frac 1 {n}\right) \]
(we have used the fact that $d_n = o(1/n)$), and so $d_n \sim 1/4n^2$. 

This process can be repeated for $j=3,4,\dots$ to find the  predicted expansion of $c_{n-1}/c_n$.
\end{proof}

%
% We have then 
%$$(n-1)\,\sum_{j=3}^{n-3} c_{j}\,c_{n-j} \leq 4\,(n-8)\,(n-1)\,c_{n-3} + 8 (n-1) \, c_{n-3} \leq 4\,n^2\,c_{n-3}.$$ 
%But by \eqref{eq:boundratio}, we have $c_{n-3} < c_{n-1}/((2n - 5)\,(2n-3))$, hence
%%$$(n-1)\,\sum_{j=3}^{n-3} c_{j}\,c_{n-j} = \mathcal O \left(n^{-1}\,c_{n-1}\right).$$
%
%Let $h_n$ be the number $c_n/c_{n-1}-2\,n$.  If we extract   the terms indexed by $k \in \ens{1,2,n-2,n-1}$ and divide the two sides by $c_{n-1}$, then Equation \eqref{eq:rec} can be rewritten as
%$$h_n = - 2 + 2 (n-1) \frac{c_{n-2}}{c_{n-1}} + \sum_{j=3}^{n-3} \frac{c_{j}\,c_{n-j}}{c_{n-1}},$$
%or
%$$h_n = - 2 + 2 (n-1) \frac 1 {2\,n + h_{n-1}-2} + \mathcal O \left(n^{-1} \right).$$
%Since Equation  \eqref{eq:boundratio} implies $-1 < h_n < 0$, we obtain $h_n = -1 + \mathcal O \left(n^{-1}\right).$
%
%The next terms in the asymptotic expansion of $c_n/c_{n-1}$ can be similarly found: extract the first terms in the sum of Equation \eqref{eq:rec}, divide by $c_{n-1}$, bound what is left thanks to Equation~\eqref{eq:gros}, express every $c_{n-k}/c_{n-1}$ in terms of the ratios $c_{n-j-1}/c_{n-j}$, and calculate an asymptotic estimate of the obtained equation. For instance, the next asymptotic term is $3/(2n)$:
%$$ \frac{c_n}{c_{n-1}} = 2n - 1 + \frac 3 {2\,n} + o \left( \frac 1 n \right).$$
%
%The estimate stated in the proposition is then obtained by taking the inverse of the previous expansion.
%\end{proof}

\section{Log expansions of the solution of the Dyson-Schwinger equation}
\label{sec:llterms}

\subsection{Context}
In this section, we show how we can deduce from Theorem~\ref{theo:MYchord} asymptotic properties on the \textit{log expansions} in quantum field theory. Let us explain first what is a log expansion.

 Suppose we have an expansion with the following form
\[
G(x,L) = 1 + \sum_{i \geq 1}\sum_{j \geq i} a_{i,j} \, L^i\,x^j.
\]
The particular Dyson-Schwinger equations we are interested in have their solution in this form as do a broad class of perturbative expansions in quantum field theory.  Then, rather than thinking of the sum first as an expansion in one of the variables with coefficients which are series in the other variables, we can take an expansion which takes variables together.

Specifically, we can write the expansion as
\[
G(x,L) = \sum_{k\geq 0} \sum_{i\geq 0} a_{i,i+k}\,(Lx)^i\,x^k.
\]
The $k=0$ part of this sum, namely the terms of $G(x,L)$ where the powers of $L$ and $x$ are the same, is known as the \emph{leading log expansion}, the $k=1$ part of this sum, namely the terms of $G(x,L)$ where the power of $x$ is one more than the power of $L$ is known as the \emph{next-to-leading log expansion}.  The $k=2$ is known as the \emph{next-to-next-to-leading log expansion} and so on.

This leading log language comes from the fact that $L$ is the logarithm of some appropriate energy scale, while $x$ is the coupling constant which is treated as a small parameter.  So the leading log expansion captures the maximal powers of $x$ relative to the powers of the energy scale, and so is in an important sense the leading term.  The next-to-leading log expansion is the next part; it is suppressed by one power of $x$, and so on.

Furthermore the full log expansion is algebraically and analytically meaningful in the sense that the contributions of larger primitive graphs and new (presumably) transcendental numbers appear further out in the next-to-next-to\dots hierarchy.  We see this manifested in our results, but it is a much more general physical fact (compare \cite{KKllog}).

In view of \eqref{sol}, the leading log expansion  for the Dyson-Schwinger equations is
\[
  -\sum_{\substack{C\\b(C) = |C|}} \frac{(-Lx)^{|C|}}{|C|!} f_0^{|C|}
\]
while the next-to-leading log expansion is
\begin{multline}
 -\sum_{\substack{C\\b(C) = |C|-1}} \frac{(-Lx)^{|C|-1}x}{(|C|-1)!} f_{0}f_0^{|C|-2}f_1
-\sum_{\substack{C\\b(C) = |C|}} \frac{(-Lx)^{|C|-1}x}{(|C|-1)!} f_{1}f_0^{|C|-1} \\
 = -\sum_{\substack{C\\b(C) \geq |C|-1}} \frac{(-Lx)^{|C|-1}x}{(|C|-1)!} f_0^{|C|-1}f_1 \label{eq:ntlle}
\end{multline}
In general the next-to${}^i$-leading log expansion is
\[
 -\sum_{\substack{C\\b(C) \geq |C|-i}} \frac{(-Lx)^{|C|-i}x^i}{(|C|-i)!} f_{b(C)-|C|+i}f_0^{|C|-k}\prod_{j=2}^kf_{t_j-t_{j-1}}
\]
where $b(C) = t_1 < t_2< \dots < t_k$ are the terminal chords of $C$.
We switch the signs from now on, both overall and of $L$, because they are the result of the conventions of \cite{Ymem} and perhaps not actually a good choice.

All this suggests that it is worthwhile to study connected diagrams $C$  such that $b(C) \geq |C|-i$, where $i$ is fixed. The present section continues by establishing numerous enumerative and asymptotic results concerning these diagrams.

\subsection{Recurrence equations} \label{Bk sub sec}

We begin by an induction that characterizes the number of connected diagrams $C$ of size $n$ such that every terminal chord has index between $n-k$ and $n$.

\begin{prop} Fix $k \geq 0$. For $n > 0$, let $b_{n,k}$ be the number of connected diagrams $C$ with $n$ chords such that $b(C) \geq n - k$, and $c_n = b_{n,n-1}$ the number of connected diagrams with $n$ chords.
For every $k\geq 0$, we have $b_{1,k} = 1$  and for $n \geq 2$
\begin{equation}
b_{n,k} = (2\,n-3)\,b_{n-1,k} + \sum_{i=1}^{\min(k,n-2)} (2\,i-1)\,c_i\,b_{n-i,k-i}. \label{eq:bnk}
\end{equation}
\label{prop:rec}
\end{prop}

This recurrence relation enables to compute the first values of $b_{n,k}$:
\begin{eqnarray*}b_{{1,0}}=1,b_{{2,0}}=1,b_{{3,0}}=3,b_{{4,0}}=15,b_{{5,0}}=
105,b_{{6,0}}=945,\dots\\
b_{{1,1}}=1,b_{{2,1}}=1,b_{{3,1}}=4,b_{{4,1}}=23,b_{{5,1}}=
176,b_{{6,1}}=1689,\dots\\
b_{{2,1}}=1,b_{{2,2}}=1,b_{{3,2}}=4,b_{{4,2}}=27,b_{{5,2}}=
221,b_{{6,2}}=2210
,\dots 
\end{eqnarray*}

\begin{proof} Equation \eqref{eq:bnk} is derived from a specific decomposition of the connected chord diagrams, which we are going to describe, and which is illustrated by Figure~\ref{fig:rec}. 
%Remark that this recursion is different from the ones existing in the literature; it 
This recursion has a good transcription in terms of exponential generating functions (see Proposition~\ref{prop:GFrec}).

\begin{figure}[ht!]
\begin{center}
\includegraphics[width=0.6\textwidth]{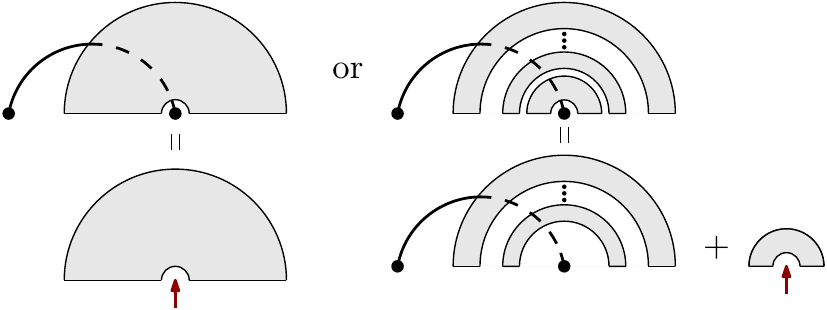}
\end{center}
\vspace*{-0.5cm}
\caption{Illustration of the decomposition from the proof of Proposition~\ref{prop:rec}. }
\label{fig:rec}
\end{figure}

Let $C$ be a connected chord diagram of size $n$ such that $b(C) \geq n - k$. If we remove the root chord of $C$, two exclusive possibilities can occur:
\begin{itemize}
\item\textbf{ The obtained diagram $C'$ is still connected.} In this case, $C'$ has $n-1$ chords and we have $ b(C') = b(C)-1$, and so  $ b(C') \geq n - 1 - k$. Moreover, to recover the diagram $C$, we need (and it is sufficient) to remember the position of the right endpoint of the root chord of $C$. Since $C'$ has $n-1$ chords, there are $2 (n-1) - 1$ possible positions. That is why the number of such diagrams $C$ is given by $(2\,n-3)\,b_{n-1,k}$.
\item  \textbf{We obtain several connected components $\boldsymbol{C_1,\dots,C_s}$.} We denote by $i$ the number of chords in $C_s$, and by $C'$ the (connected) diagram obtained by removing $C_s$ from $C$. Diagram $C'$ has $n-i$ chords, and the position of the first terminal chord has remained unchanged, hence $b(C') \geq (n-i) - (k-i)$. Observe then that it is possible to recover $C$ from $C'$ and $C_s$ provided the position of the root chord through $C_s$ (there are $(2i-1)$ such possible positions). The number of such diagrams $C$ is thus $ (2\,i-1)\,c_i\,b_{n-i,k-i}$.
\end{itemize}
The conjunction of these two cases infers Equation \eqref{eq:bnk}.
\end{proof}

Remark that for $k=0$, Recurrence \eqref{eq:bnk} is simply $b_{n,0} = (2\,n-3)\,b_{n-1,0}$. This provides a nice formula (and a combinatorial proof!) for the numbers $b_{n,0}$, which correspond to the numbers of connected chord diagrams with exactly one terminal chord (the last chord of a diagram is necessarily terminal).

\begin{cor} \label{cor:oneterm} The number of connected diagrams with $n$ chords and only one terminal chord is $(2\,n-3)!!$.
\end{cor}

The recurrence relation of Proposition \ref{prop:rec} can be transformed in an effective way to compute the exponential generating functions of the numbers $b_{n,k}$:

\begin{prop} \label{prop master eq} Let $B_k(z)$ be the exponential generating function  of the connected chord diagrams $C$ such that $b(C) \geq n - k$. For every pair of integers $i,k$, we consider an $i$th antiderivative\footnote{It is possible to define uniquely this antiderivative by setting for example $\pd {^j B^{[i]}_k}{z^j}(0) = 0$ for $j \in \ens{0,\dots,i-1}$, but in practice, it is more convenient to take any antiderivative we find.}  $B^{[i]}_k$ for $B_k$, that is, a function  $B^{[i]}_k$ such that its $i$th derivative is equal to $B_k$. There exists a constant $\beta_k$ and a polynomial $P_k$ of degree $k$ such that
\begin{equation}
B_k(z) = \sqrt{1-2\,z} \left(\beta_k + \sum_{i=1}^{k}  (2\,i-1)\,c_i \, \int_0^z (1-2x)^{-3/2} B^{[i-1]}_{k-i}(x)\,dx   \right) + P_k(z)
\label{eq:GFrec}
\end{equation}
where $c_i$ is the number of connected diagrams with $i$ chords.
\label{prop:GFrec}
%The values of $\beta_k$ and $P_k$ can be calculated by plugging in the above relation the $(k+1)$ first terms of the  generating functions $B_k(z)$.
\end{prop}
\begin{proof} This proof can be divided into two steps. First, we translate \eqref{eq:bnk} in terms of the functions $B_{k-i}$ with $i \in \ens{0,\dots,k}$, which gives a first order differential equation in $B_k$. Then, we simply solve this differential equation.

 Let $n > k+1$. Dividing \eqref{eq:bnk} by $(n-1)!$ and writing $2n-3$ as $2(n-1) - 1$ induces that
\begin{equation}
 \frac{b_{n,k}}{(n-1)!} - 2 \, \frac{b_{n-1,k}}{(n-2)!} + \,\frac{b_{n-1,k}}{(n-1)!} = \sum_{i=1}^{k} (2\,i-1)\,c_i\,\frac{b_{n-i,k-i}}{(n-1)!}.
 \label{eq:bndiv}
\end{equation}
Observe in all generality that the series of general term $a_{n-i-1}\,z^{n-1}/(n-1)!$ is the $(-i)$th derivative of the exponential generating function of the sequence $a_n$ if $i$ is non-positive, and an $(i-1)$th antiderivative if $i$ is positive. We then recognize in  \eqref{eq:bndiv} the coefficients of $z^{n-1}$ in $ \pd {B_k} z, z\pd {B_k} z, B_k, B^{[i-1]}_{k-i}$, respectively. Thus \eqref{eq:bndiv} can be translated by
\[ (1-2\,z) \pd {B_k} z + B_k = \sum_{i=1}^{k} (2\,i-1)\,c_i\,B^{[i-1]}_{k-i} + Q_k(z), \]
where $Q_k(z)$ is a polynomial of degree $k$ whose presence is due to the fact that the first coefficients of $B^{[i-1]}_{k-i}$ can vary, but also because \eqref{eq:bndiv} holds only for $n > k+1$.

We can solve this differential equation quite straightforwardly: we divide by $(1-2z)^{3/2}$ both sides and recognize from the left side the derivative of $(1-2z)^{-1/2} B_k$. Integrating this equation then leads to \eqref{eq:GFrec}.
(We have set $P_k(z) = \sqrt{1-2\,z} \int_0^z (1-2x)^{-3/2} Q_k(x) dx$. Some easy calculus shows that $P_k$ is  also a polynomial of degree $k$.)
\end{proof}

\noindent \textbf{Remark 1.} It is simple to compute the series $B_k(z)$ by recursion thanks to Formula \eqref{eq:GFrec}. The method is the following: for each $i$, we begin by compute the antiderivatives $B^{[i-1]}_{k-i}$, then plug them into \eqref{eq:GFrec}, evaluate the formula and then eliminate $\beta_k$ and $P_k(z)$ thanks to the first values of $B_k(z)$ given by Proposition \ref{prop:rec}. We thus obtain:
\begin{eqnarray*}
B_0(z) = & 1 - \sqrt{1-2\,z}, \\
B_1(z) = & 1 + z+ \frac 1 2\,\sqrt {1-2\,z} \,\ln  \left( 1-2\,z \right) -\sqrt {1-2\,z}, \\
B_2(z) = & \left(\frac {\ln  \left( 1-2\,z \right)} 2  - \frac{ \ln  \left( 1-2\,z \right)^{2}
 }  8 + z-3 \right)\,\sqrt {1-2\,z} + 3  -2\,z+ \frac {{z}^{2}} 2.
\end{eqnarray*}
It is important to notice that the method is automatic.  In this regard, a \texttt{maple} file is available along with the \texttt{arXiv} version of this paper.

\noindent \textbf{Remark 2.} The foregoing gives information about the generating function of connected diagrams $C$ such that $b(C) \geq n - k$ but nothing about the distribution of $f_{b(C)-i}f_0^{|C|-k}\prod_{j = 2}^{k} f_{t_j-t_{j-1}}$ in the leading-log coefficients. However it is easy to adapt the same approach to enumerate diagrams to specific cases where the $t_j-t_{j-1}$ are fixed.

%\begin{prop} Let $a_n[\sigma_1,\dots,\sigma_{k-1}]$ be the number of connected diagrams with $n$ chords where $b(C) = t_1 < \dots < t_{k-1} < t_k = n$ are the positions of the terminal chords and such that $\sigma_j = t_{j+1}-t_j$ for $j \in \{1,\dots,k-1\}$. Then if $k > 1$, we have for $n > \sum_{j=1}^{k-1} \sigma_j$:
%\[a_n[\sigma_1,\dots,\sigma_{k-1}] = (2\,n-3)\,a_{n-1}[\sigma_1,\dots,\sigma_{k-1}] + (2\,\sigma_{k-1}-1)!!\,a_{n-\sigma_{k-1}}[\sigma_1,\dots,\sigma_{k-2}].\]
%If $A_{[\sigma_1,\dots,\sigma_{k-1}]}(z)$ denotes $\sum_{n \geq 0} a_n([\sigma_1,\dots,\sigma_{k-1}]) z^n/n!,$ and $A_{[\sigma_1,\dots,\sigma_{k-2}]}^{[\sigma_{k-1}-1]}$ any $(\sigma_{k-1}-1)$th antiderivative of $A_{[\sigma_1,\dots,\sigma_{k-2}]}$, then there exist a constant $\beta$ and a polynomial $P$ of degree $\sum_{j=1}^{k-1} \sigma_{j}$ such that
%\begin{equation*}
%A_{[\sigma_1,\dots,\sigma_{k-1}]}(z) = \sqrt{1-2\,z} \left(\beta +  (2\,\sigma_{k-1}-1)!! \, \int_0^z (1-2x)^{-3/2} A_{[\sigma_1,\dots,\sigma_{k-2}]}^{[\sigma_{k-1}-1]}(x)\,dx   \right) + P(z).
%\end{equation*}
%If $k=1$, then $a_n[\varnothing ]=(2\,n-3)!!$ is the number of connected diagrams with only one terminal chord, and $A_{[\varnothing ]}(z)=1- \sqrt{1-2\,z}$ its generating function (cf Corollary \ref{cor:oneterm}).
%\end{prop}

\noindent  \textbf{Example.} Let us consider $A(z)$ the exponential generating function of connected diagrams such that the only terminal chords are the third to last and last ones (i.e. connected diagrams $C$ such that  $b(C)=t_1=|C|-2$ and $t_2=|C|)$), and let us use the same decomposition as in the proof of Proposition~\ref{prop:rec}. Removing the root chord in such diagrams leads to two possibilities:
\begin{itemize}
\item The resulting diagram has only one component; starting from the end, the positions of the terminal chords do not change.
\item It has several components. Since each component necessarily has at least one terminal chord, the number of components is exactly two: the top component has only one terminal chord, and the bottom component is the only connected diagram with $2$ chords (with $3$ possibilities of insertion for the root chord). 
\end{itemize}
This consideration leads to the recurrence 
\[a_n = a_{n-1} + 3\,(2\,n-7)!!,\]
where $a_n$ is the number of connected diagrams with $n$ chords such that the only terminal chords are the third to last and last ones. By the same process than previously, we can then prove that 
\begin{equation*}A(z) = (z - 1)\, \sqrt {1-2\,z}+ \frac{{z}^{2}}2-2\,z+1.
\end{equation*}
%\begin{proof}(Sketch.)  Using the same decomposition as the proof of Proposition~\ref{prop:rec}, we observe that if the deletion of the root chord leads to several connected components, then the last one (starting from top to bottom) has necessarily size $\sigma_{k-1}$ and has only terminal chord 
%\end{proof}
%
%
%Once again, the process can be automatized. 
%
%

In all generality, similar recursions exist for diagrams where the gaps between the terminal chords $t_2-t_1,\dots,t_k-t_{k-1}$ are given, but equations are more tedious to state (although the method will fundamentally remain the same). If the reader would like to compute such  generating functions, a procedure is written in the aforementioned  \texttt{maple} file.

\subsection{Asymptotic behaviour}\label{subsec asymp}

Now that we have stated how to compute the numbers $b_{n,k}$, we are interested by their a\-symptotic behaviour. The following theorem gives the asymptotic estimate.

\begin{theo} The number of connected diagrams $C$ with $n$ chords such that $b(C) \geq n - k$ is asymptotically equivalent to
\[\frac 1 {\sqrt{\pi}\,2^{k+1}\,k!} \frac {\ln(n)^k\,2^n\,n!} { n^{3/2} }.\]
\label{theo:bnk}
\end{theo}

\begin{proof} We just apply the transfer theorem (Theorem~ \ref{theotransfer}) to the exponential generating function $B_k$ characterized by the following lemma. Indeed this lemma shows that 
$B_k(z) - B_k\left(1/2\right)$ is equivalent to $\frac{(-1)^{k+1} }{2^{k}\,k!} \sqrt{1-2\,z} \ln(1-2z)^k$ when $z \rightarrow 1/2$.
\end{proof}

\begin{lem} The exponential generating function $B_k$ from Proposition~\ref{prop:GFrec} is a polynomial in terms of $\sqrt{1-2z}$ and $\ln(1-2z)$:
\begin{equation}
B_k(z) = B_k\left(\frac 1 2\right) +  \sqrt{1-2\,z}\,Q_k\left(\sqrt{1-2\,z},\ln(1-2z)\right),
 \label{eq:Bkis}
\end{equation}
where $Q_k(x,y)$ is a polynomial of degree $k$ in $y$ such that the coefficient of $y^k$ is
%\[B_k(z) = B_k\left(\frac 1 2\right) + \sum_{j=0}^k p_{k,j}(1-2z) \sqrt{1-2z} \ln(1-2z)^j  + \sum_{\ell=2}^{2k} \lambda_{k,\ell} (1-2z)^{\ell/2},\]
%where $p_{k,j}$ is a polynomial of degree at most $\max(0,j-2)$, and $\lambda_{k,ell}$ are constant. Moreover, $p_{k,k}$ is a constant polynomial and equals
 \[[y^k] \, Q_k(x,y) = \frac{(-1)^{k+1} }{2^{k}\,k!}\]
 (there is no term in $x^i$, with $i \geq 1$).
\end{lem}

\begin{proof}This lemma can be shown by induction on $k$.

For $k=0$, we have seen that $B_0(z)= 1 - \sqrt{1-2\,z}$.

For $k \geq 1$, we have to check that the statement of the lemma is compatible with \eqref{eq:GFrec}. First, using the induction hypothesis, we verify that any $(i-1)$th antiderivative $B^{[i-1]}_{k-i}$ of $B_i$ for $i \in \ens{1,\dots,k}$  is a polynomial in  $\sqrt{1-2z}$ and $\ln(1-2z)$ such that the degree in $\ln(1-2z)$ does not exceed $k-i$. 
(To compute an antiderivative of $(1-2\,z)^{i/2} \ln(1-2\,z)^j$, we repeatedly  integrate by parts using the equality \[(i+2) \int_0^z (1-2\,x)^{i/2} \ln(1-2\,x)^j dx = -(1-2\,z)^{i/2+1} \ln(1-2\,z)^j  - 2\,j \int_0^z (1-2\,x)^{i/2} \ln(1-2\,x)^{j-1} dx\]
until the degree in $\ln(1-2\,x)$ in the integrand reaches 0.)
  Using \eqref{eq:GFrec}, it is then not hard to check that $B_k(z)$ can be put into the form \eqref{eq:Bkis}. If we search in \eqref{eq:GFrec} for what could contribute to the term in $y^k$ in $Q_k(x,y)$, we realize that the only possibility comes from the monomial  $\sqrt{1-2z}\,\ln(1-2z)^{k-1}$ in $B_{k-1}(z)$. Indeed, we can observe that
\[  \sqrt{1-2z} \int_0^z (1-2x)^{-3/2} \times \sqrt{1-2x}\,  \ln(1-2\,x)^{k-1}   \,dx = \frac{-1}{2k} \sqrt{1-2z}\,\ln(1-2z)^k.\]
  
  By recurrence, we know that the coefficient of $\sqrt{1-2z}\,\ln(1-2z)^{k-1}$ in $B_{k-1}(z)$ is $\frac{(-1)^{k} }{2^{k-1}\,(k-1)!}$, so the coefficient of $\sqrt{1-2z}\,\ln(1-2z)^k$ in $B_{k-1}(z)$ must be $\frac{-1}{2k} \times \frac{(-1)^{k} }{2^{k-1}\,(k-1)!} = \frac{(-1)^{k+1} }{2^{k}\,k!}$.
%can be put under the form
%\[B^{[i-1]}_{k-i}(z) = B^{[i-1]}_{k-i}\left(\frac 1 2\right) + \sum_{j=0}^{k-i} p_{k,i,j}(1-2z) \sqrt{1-2z} \ln(1-2z)^j  + \sum_{\ell=2}^{2k} \mu_{k,i,\ell} (1-2z)^{\ell/2},\]
%where $q_{k,i}$
\end{proof}

Once again, the foregoing does not give any information about the asymptotic distribution on the terminals chords. However we can recover it by repeating the same reasoning for the number of connected diagrams such that only the last $k$ chords for the intersection order are terminal, and observe that the asymptotic behaviour is identical.

\begin{theo} The number $o_{n,k}$ of connected diagrams $C$ with $n$ chords such that the only terminal chords are the last $k$ chords is asymptotically equivalent to
\[\frac 1 {\sqrt{\pi}\,2^{k+1}\,k!} \frac {\ln(n)^k\,2^n\,n!} { n^{3/2} }.\]
\label{theo:onk}
\end{theo}
\begin{proof} (Sketch.) Using the decomposition of the proof of Proposition~\ref{prop:rec}, we find that the numbers $o_{n,k}$ satisfy
\[o_{n,k} = (2\,n-3)\,o_{n-1,k} + o_{n-1,k-1}. \]
This recurrence relation can be then translated into the differential equation 
\[ (1-2\,z) \pd {O_k} z + O_k = O_{k-1} + \widetilde Q_k(z), \]
where $O_k(z) = \sum_{n \geq 0} o_{n,k}\,z^n/n!$ and $\widetilde Q_k$ a polynomial of degree $k$. Its solutions can be put into the form
\[
O_k(z) = \sqrt{1-2\,z} \left( \widetilde\beta_k + \int_0^z (1-2x)^{-3/2} O_{k-1}(x)\,dx   \right) + \widetilde P_k(z),
\]
where $\widetilde \beta_k$ is a constant and $\widetilde P_k(z)$ a polynomial. By recurrence, we can then prove that $O_k$ is a polynomial in $\sqrt{1-2z}$ and $\ln(1-2z)$ such that the contributing term for the singularity analysis is $\frac{(-1)^{k+1} }{2^{k}\,k!} \sqrt{1-2z} \ln(1-2z)^k$. We recover the expected asymptotic regime by the transfer theorem.
\end{proof}

The consequence of the similarity between Theorem~\ref{theo:bnk} and Theorem~\ref{theo:onk} will be described in the next subsection.

\subsection{Application to the log expansions}

The leading log expansion is particularly simple because it only counts chord diagrams where only the last chord is terminal.  By Corollary \ref{cor:oneterm}, these are easy to count, and the monomial in the $f_i$ is simply a power of $f_0$. Therefore it suffices to understand $B_0(z)$.  Specifically, the leading log expansion is
\begin{equation}\label{leading log eq}
   B_0(Lxf_0) = 1 - \sqrt{1-2Lxf_0}
\end{equation}

The next-to-leading log expansion is not too difficult either.  
%We need to consider chord diagrams where the last chord is the only terminal and chord diagrams where the second last chord is also terminal.  The former type of diagram contribute monomials of the form $f_1f_0^{|C|-1}$ while the latter contribute $f_{0}f_0^{|C|-2}f_{1}$ which is the same monomial obtained differently.  
By \eqref{eq:ntlle} it suffices to understand $B_1(z)$.  Note, however, that the power of $Lxf_0$ is $|C|-1$, and we are dividing by $(|C|-1)!$, so the next-to-leading log expansion is actually given in terms of the derivative of $B_1(z)$.  Specifically, the next-to-leading log expansion is
\begin{equation}\label{next to leading log eq}
  \frac{d}{dz}B_1(z)|_{z = Lxf_0} xf_1 = xf_1\left(1 + \frac{1}{\sqrt{1-2Lxf_0}}\ln\left(\frac{1}{\sqrt{1-2Lxf_0}}\right)\right)
\end{equation} 

The next-to-next-to-leading log expansion is a bit more complicated.  Here we are considering any chord diagram with $b(C)\geq |C|-2$.  Now there are different possible monomials.  If all of the last three chords are terminal then we get $f_0f_0^{|C|-3}f_1^2$ while if only the last and the third last are terminal we get $f_0f_0^{|C|-2}f_2$.  If only the last two chords are terminal we get $f_1f_0^{|C|-2}f_1$ and if only the last chord is terminal we get $f_2f_0^{|C|-1}$.  All together two different monomials appear, $f_2f_0^{|C|-1}$ in the case that either the last and third last or just the last are terminal, and $f_1^2f_0^{|C|-2}$ in the case that either the last two or the last three are all terminal.  In all cases we will need to take two derivatives since the powers and factorials are in terms of $|C|-2$ for the next-to-next-to leading log expansion rather than in terms of $|C|$ for the exponential generating functions $B_k$.  
			
			Using the $A(z)$ from the example in Subsection~\ref{Bk sub sec} we can calculate the next-to-next-to-leading log expansion explicitly:
			\begin{multline*}
			   x^2f_2f_0\left(\frac{d^2}{dx^2}(A(z)+B_0(z))\right)\bigg|_{z=Lxf_0} + x^2f_1^2\left(\frac{d^2}{dx^2}(B_2(z)-A(z)-B_0(z))\right)\bigg|_{z=Lxf_0} \\
			   = x^2f_0f_2\left(1+ \frac{3Lxf_0}{(1-2Lxf_0)^{3/2}}\right) + \frac{x^2f_1^2(\ln(1-2Lxf_0)-4)\ln(1-2Lxf_0)}{8(1-2Lxf_0)^{3/2}}.
			\end{multline*}
Latter log expansions work similarly.

Let us compare these results to the results of Kr\"uger and Kreimer in \cite{KKllog}.  Their methods are also combinatorial but are quite different.  They are based on words on the alphabet of primitive graphs operated on by shuffle and Lie bracket.  Despite these differences we are both modelling the same underlying physics, so our answers should agree on the common domain of applicability.  
			
			Our results correspond to their Yukawa case with only one primitive.  In fact we deal with any Dyson-Schwinger equation with this shape.  They could also do so, but chose to only make the Yukawa and QED examples explicit.  On the other hand their work is more general in that they deal with any number of primitives in the Yukawa and QED example.  In view of \cite{HYchord} our results should also generalize to any number of primitives and to Dyson-Schwinger equations of QED shape and other shapes (corresponding to different $s$ parameters in the setup of \cite{Ymem}).  This will be worked out in the future.

Our leading log calculations are, as they should be, identical (compare \eqref{leading log eq} to Equation 221 of \cite{KKllog}).   The next-to-leading log (compare \eqref{next to leading log eq} and Equation 227 of \cite{KKllog}) are very similar.  First in our case we are not considering a new primitive graph at 2 loops, which in the language of Kr\"uger and Kreimer would say that $\Phi_R(\Gamma_2)= 0$.  
%Our methods should generalize to multiple primitives, as mentioned above, but this remains to be done.  
The second thing to notice is a spurious $1$ in the derivative of $B_1(z)$. Its presence is due to different boundary conditions.  They explicitly set their generating function to have no constant term (see the line after Equation 146) while our boundary conditions are determined by the chord diagrams: this particular 1 corresponds to the connected chord diagram with two chords.  Finally, note that they have a more complicated expression in place of our $f_1$.  In both cases this number is the new period.  Kr\"uger and Kreimer call it $\Theta(a_1, a_1)$; they note that it cannot be canonically identified with a single Feynman graph.  From our perspective we see it naturally as the next term in the expansion for the original primitive. 

Turning to the next-to-next-to-leading log expansion, we again see that our solution is built on of the same kinds of pieces as theirs.  Their greater generality shows up more strongly here as our solution is strictly simpler.  We also see more clearly at this level how our different perspectives result in different characterizations of the new primitives.

What are the benefits and disadvantages of our techniques compared to the techniques of Kr\"uger and Kreimer?  Both methods have a combinatorially derived master equation that determines everything.  For us this would be the recursive decomposition of the previous sections -- we did not write it out as an integral or differential equation in general (only in the important special case of Proposition \ref{prop master eq}), but the example in Subsection~\ref{Bk sub sec} illustrates how it works in general.  For both groups the master equation is not fully explicit.  In Kr\"uger and Kreimer's set up this manifests itself in the dependence on matrix bracket coefficients for which it is unclear how automatically or rapidly they can be computed.  The lack of explicitness however has a different flavour in each case coming from the different combinatorial objects.  %Furthermore, our approach leads to fully explicit special cases as in Proposition \ref{prop master eq} and the asymptotics of Subsection \ref{subsec asymp}.

Our technique also differs in how it indexes the periods which contribute to the expansions. Kr\"uger and Kreimer tie them to individual graphs where possible and treat the others, coming from their $\Theta$-expressions on the same level.  We do not give these periods individual meanings but see them as coming from later terms in the expansion our one primitive; this makes our periods less combinatorial, but they are organized into tidy monomials so one can better see the different pieces that build them.  Here both techniques have advantages and one would hope to play them off each other to get an even better understanding.  The same can be said about the different underlying combinatorial frameworks -- the physics is described both by our chord diagrams and by their words and it is not obvious, but is potentially useful, that both these objects describe the same underlying structures.  

\medskip

Finally, we can consider the significance of the asymptotic results of Subsection~\ref{subsec asymp} to the log expansions.  Here something very interesting happens.  The chord diagrams where the only terminal chords are the last $k$ chords dominate completely in the sense that as $n\rightarrow \infty$ almost all chord diagrams with $n$ chords and $b(C)\geq n-k$ have the last $k$ chords terminal.  What this means is that provided $F(\rho)$ is not outrageous (eg the $f_i$ are bounded) we should expect the chord diagrams with the last $k$ chords terminal to completely determine the asymptotic behaviour of the next-to${}^k$-leading log expansion.  That is, the next-to${}^k$-leading log expansion should behave as if all chord diagrams contribute the monomial $f_0^{|C|-k+1}f_1^{k-1}$, so the asymptotic behaviour of the next-to${}^k$-leading log expansion is given by
\begin{equation}
\frac{d^k}{dz^k}B_k(z)|_{z = Lxf_0}f_1^{k-1}
\label{eq:f0f1}
\end{equation}
This is nice for two reasons.  First it says that the other $f_i$ are not playing a significant role asymptotically -- this means only two numbers, $f_0$ and $f_1$, are controlling their asymptotic behaviours.  Second the master equation to generate the $B_k(z)$ is fairly simple and can be computed fully automatically.  This is much simpler than the situation for chord diagrams with specific gap patterns as we calculated for the next-to-next-to-leading log expansion and contains no mysteries which require human intervention to compute.

\section{Statistics on terminal chords}
\label{sec:stat}

\subsection{Statement of the meta-theorem and examples}

In this section, we study several statistics concerning terminal chords in connected diagrams, such like its numbers, the number of terminal chords that are consecutive for the intersection order, etc. We establish a meta-theorem that shows that a lot of random variables on connected chord diagrams have a Gaussian limit law with logarithmic variance.

Before stating this theorem, we need to define three subsets of connected chords diagrams based on the shape of the diagram obtained by removing the root chord:
\begin{align*}
\mathcal C_1 = & \left\{C \ |\ \textrm{Removing the root chord from }C\textrm{ leads to a unique component }C_1.\right\}, \\
\mathcal C_2 = & \left\{C \ |\ \textrm{Removing the root chord
% from }C\textrm{ 
leads to a chord and a component }C_2\textrm{, top to bottom.}\right\},\\
\mathcal C_3 = & \left\{C \ |\ \textrm{Removing the root chord
% from }C\textrm{ 
leads to a component }C_3\textrm{ and a chord, top to bottom.}\right\}.
\end{align*}
These three subsets are illustrated by Figure~\ref{fig:subsets}.
% The set $\mathcal C_1$ is the set of connected chord diagrams in which the diagram obtained by removing the root chord has a unique connected component $C_1$; the set $\mathcal C_2$ comprises every diagram such that the previous resulting diagram is formed of an isolated chord and a connected component $C_2$ below it; $\mathcal C_3$ comprises every diagram such that the resulting diagram is formed of an isolated chord and a connected component $C_3$ above it.

\begin{figure}[ht!]
\begin{center}
\includegraphics[width=0.8\textwidth]{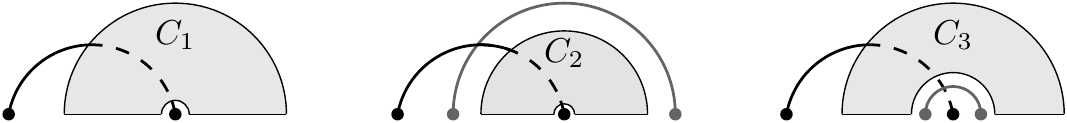}
\end{center}
\caption{From left to right, a schematic representation of an element of $\mathcal C_1$, $\mathcal C_2$, $\mathcal C_3$.}
\label{fig:subsets}
\end{figure}

\begin{theo} \label{theo:bigone} Let $\lambda_1$, $\lambda_2$, $\lambda_3$ be three integers (not all equal). Consider a function $f$ on connected chord diagrams such that for every $C \in \mathcal C_i$ with $i = 1,2,3$, $f(C) = f(C_i) + \lambda_i$ (see above for the definition of $\mathcal C_1, \mathcal C_2, \mathcal C_3$). If $X_n$ denotes a random chord diagram of size $n$ under the uniform distribution, then $f(X_n)$ is a random variable such that $\dfrac{f(X_n) - \lambda_1\,n - \mu\,\ln(n)}{\sigma \sqrt{\ln(n)}}$ converges in distribution to a standard Gaussian law, where
\[\mu =  \frac {\lambda_2} 2 + \frac {\lambda_3} 2 - \lambda_1,\quad \sigma^2 = \frac {(\lambda_2-\lambda_1)^2} 2 + \frac {(\lambda_3-\lambda_1)^2} 2.\]
\end{theo}

Among other things, this theorem implies that the relevant diagrams under the uniform distribution are those whose recursive decomposition only uses diagrams from $\mathcal C_1, \mathcal C_2$ and $\mathcal C_3$. The other diagrams are asymptotically negligible (the proof of this theorem just uses this fact).

Let us illustrate Theorem \ref{theo:bigone} with some examples.  If we denote by $T_n$ the random variable on connected diagrams with $n$ chords that counts the terminal chords, we can see that $T_n = f(X_n)$  where $f$ and $X_n$ are described in the statement of the theorem with $\lambda_1 = 0$ and $\lambda_2 = \lambda_3 = 1$. (Only diagrams from $\mathcal C_2$ and $\mathcal C_3$ have a decomposition which induces terminal chords -- the terminal chords correspond to the dark-grey ones in Figure~\ref{fig:subsets}.) Consequently, we have the following corollary.
\begin{cor}
 The random variable $T_n$ for the number of terminal chords asymptotically has a Gaussian limit law with a mean and a variance equivalent to $\ln(n)$.
 \label{cor:tn}
\end{cor}

Now let us consider $G_{1,n}$, the random variable on connected diagrams with $n$ chords that counts the pairs of terminal chords that are adjacent in the intersection order. Equivalently, $G_{1,n}$ counts the number of terminal chords $c$ such that the chord that precedes $c$ in the intersection order is also terminal. We can then notice that decompositions of diagrams from $\mathcal C_1$ and $\mathcal C_2$ do not induce such terminal chords (for the former, the only apparent chord is not terminal; for the latter; the chord that precedes the terminal chord is the root chord, which is not terminal), while decompositions for $\mathcal C_3$ do (the chord that precedes the terminal chord is the last chord of $C_3$ which is terminal -- the last chord of a connected diagram is always terminal). Therefore, we have $G_{1,n} = f(X_n)$ with $\lambda_1 = \lambda_2 = 0$ and $\lambda_3=1$, which gives the following result.

\begin{cor}
\label{cor:g1n}
 The random variable $G_{1,n}$ for the pairs of terminal chords that are adjacent for the intersection order has a Gaussian limit law with a mean and a variance  asymptotically  equivalent to $\dfrac {\ln(n)} 2$.
\end{cor}

\noindent \textbf{Remark.}
It is worth noting that the standard theory cannot be used directly. Indeed, the main obstacle is the non-analyticity of the ordinary generating functions, which \textit{a priori} prevents any use of complex analysis. For instance, if we consider $C(z,u)$ the generating function of connected diagrams, where $z$ refers to the number of chords and $u$ to the number of terminal chords, this series satisfies the differential equation (which can be established by a straighforward combinatorial specification -- see \cite{flajolet-sedgewick})
\[C(z,u) = z\,u + z\, \frac{2z\pd C z (z,u) - C(z,u)}{1 - 2z\pd C z (z,u) + C(z,u)}.\]
We can solve this non-linear differential equation (to some extent -- the solution can be implicitly defined in terms of the Whittaker functions) but it seems to be impossible to deduce anything from there.

\subsection{Proof of Theorem \ref{theo:bigone}}

First of all, remark that we can assume that $\lambda_1=0$ without any lost of generality. Indeed, we can study $\tilde f(X_n) := f(X_n)-\lambda_1\,n$ instead of $f(X_n)$. The new function $\tilde f$ satisfies the conditions of Theorem~\ref{theo:bigone} where the new set of parameters $\tilde \lambda_1, \tilde \lambda_2, \tilde \lambda_3$ is equal to $0$, $\lambda_2 - \lambda_1$, $\lambda_3 - \lambda_1$. From the rest of this subsection, we suppose $\lambda_1=0$.

\newcommand{\prob}[1]{\mathbb P \left(#1\right) }
\newcommand{\ev}[1]{\mathbb E \left(#1\right) }

Before presenting the idea of the proof of Theorem~\ref{theo:bigone}, let us state an asymptotic equation governing the probabilities $\prob{f(X_n) = k}$, that we denote shorthand $p_{n,k}$. 

\begin{lem} Let $f$ and $X_n$ as stated by Theorem~\ref{theo:bigone} with $\lambda_1=0$, and let $p_{n,k}$ denote $\prob{f(X_n) = k}$. Then, when $n$ goes to infinity,
\begin{equation}p_{n,k} = \left(1 - n^{-1} \right) \,p_{n-1,k} + \frac{n^{-1}} 2 \,\left(p_{n-2,k-\lambda_2} + p_{n-2,k-\lambda_3}\right) + O\left(n^{-2}\right).
\label{eq:pnk}
\end{equation}
\label{lem:asympteq}
\end{lem}
\begin{proof}
Under the condition $X_n \in \mathcal C_1$, the probability that $f(X_n)$ equals $k$ is $p_{n-1,k}$. Indeed, removing the root chord from $X_n$ gives a \emph{uniform} connected diagram $C_1$ with $n-1$ chords such that $f(X_n) = f(C_1)$ (since $\lambda_1=0$). Similarly, the probability that $f(X_n)$ equals $k$ under the condition $X_n \in \mathcal C_i$, with $i = 2$ or $3$, is $p_{n-2,k-\lambda_i}$. This shows that
\begin{multline*}
p_{n,k}= \prob{X_n\,\in\,\mathcal C_1} \,p_{n-1,k} + \prob{X_n\,\in\,\mathcal C_2} \,p_{n-2,k-\lambda_2}  \\ + \prob{X_n\,\in\,\mathcal C_3} \,p_{n-2,k-\lambda_3} + \prob{X_n \notin \bigcup_{i=1,2,3} \mathcal C_i \textrm{ and }  f(X_n) = k } 
\end{multline*}
We can then see that $\prob{X_n\,\in\,\mathcal C_1} = (2\,n-3)\,\frac {c_{n-1}}{c_n}$ since the number of diagrams of size $n$ in $\mathcal C_1$ is equal to the number of size of connected diagram of size $n-1$ (namely $c_{n-1}$) times the number of ways of inserting the root chord in this diagram ($2n-3$ ways to do it). By Proposition~\ref{prop:expansion}, we deduce that $\prob{X_n\,\in\,\mathcal C_1} = 1 - n^{-1} + O\left(n^{-2}\right)$. Similarly,
$\prob{X_n\,\in\,\mathcal C_2} = \prob{X_n\,\in\,\mathcal C_3} = (2\,n-5)\,\frac {c_{n-2}}{c_n} = n^{-1}/2 + O\left(n^{-2}\right).$ Finally, we deduce
\[\prob{X_n \notin \bigcup_{i=1,2,3} \mathcal C_i \textrm{ and }  f(X_n) = k } \leq 1 - \sum_{i=1,2,3} \prob{X_n\,\in\,\mathcal C_i} = O\left(n^{-2}\right),\]
which proves the lemma.
\end{proof}

Lemma~\ref{lem:asympteq} suggests that the recursive equation relating the numbers $p_{n,k}$ is easy to study (mainly because it  \textit{almost} involves polynomial coefficients), but the presence of the error term $O(n^{-2})$ makes the analysis tricky. The idea  then consists in forgetting this term and studying the sequences $\left(q_{n,k}\right)$ defined by
\begin{equation}
q_{n,k} = \left(1 - n^{-1} \right) \,q_{n-1,k} + \frac{n^{-1}}2 \,\left(q_{n-2,k-\lambda_2} + q_{n-2,k-\lambda_3}\right).
\label{eq:tildepnk}
\end{equation}
After that, we find a relation between the sequences  $\left(q_{n,k}\right)$ and the original sequence $(p_{n,k})$, which terminates the proof.

Remark that if $\left(q_{n_0,k}\right)_k$ and $\left(q_{n_0+1,k}\right)_k$ define two probability distributions (i.e. $\sum_k q_{n_0,k} = \sum_k q_{n_0+1,k} = 1$, and  $q_{n_0,k}$ and $q_{n_0+1,k}$ are non-negative for every $k$), then by a simple induction, $\left(q_{n,k}\right)_k$ also defines a  probability distribution  for \textit{all} integers $n \geq n_0$. In this case, we can define for every $n \geq n_0$ a random variable $Y_n$ such that $\prob{Y_n=k} = q_{n,k}$. The following lemma states that $Y_n$ tends to a Gaussian law.

\begin{lem} Set $n_0 \geq 0$. Let us consider $\left(q_{n,k}\right)_{n\geq n_0,k \geq 0}$ a sequence of numbers that:
\begin{itemize}
\item defines a probability distribution,
\item satisfies \eqref{eq:tildepnk} after $n \geq n_0+2$,
\item has a finite support when $n=n_0$ and $n=n_0+1$ (that is, the number of $k$ such that $q_{n,k} \neq 0$ is finite).
\end{itemize}   
%(we impose no condition until $n_0$ -- except the fact it must define a probability distribution)
 If  $Y_n$ denotes the random variable defined as $\prob{Y_n=k} = q_{n,k}$, then $\dfrac{Y_n - \mu\,\ln(n)}{\sigma \sqrt{\ln(n)}}$ converges in distribution to a standard Gaussian law, where
$\mu =  \lambda_2/2 + \lambda_3/2$ and $\sigma^2 = \lambda_2^2/2 + \lambda_3^2/2$.
\label{l:Yn}
\end{lem}

The subtlety of this lemma lies in the fact that we consider sequences  $\left(q_{n,k}\right)$ that are only defined after some fixed number $n_0$, without any initial condition. This flexibility on $n_0$ will be crucial for the final proof.

\begin{proof} We show here that the generating function of the numbers $q_{n,k}$ satisfies the hypotheses of the Quasi-Powers theorem (Theorem~\ref{theo:qp}).

\textbf{Step 1: completing the sequence.}
We first complete the sequence $\left(  q_{n,k}\right)$ so that it satisfies  \eqref{eq:tildepnk}
for \textit{every} $n \geq 0$. To do so, we define $q_{n,k}$ for $n \in \left\{0,\dots,n_0-1 \right\}$ by considering \eqref{eq:tildepnk} as a backward recurrence:
 \[ q_{n-2,k-\lambda_3} =  2n\,q_{n,k} + 2\,\left( 1 - n \right) \,q_{n-1,k} - q_{n-2,k-\lambda_2},\]
 with initial condition $q_{n,0} = 0$ for every $n \in \left\{0,\dots,n_0-1 \right\}$. (We have assumed that $\lambda_3$ is smaller than $\lambda_2$. If it is not the case, we can still swap the roles of $\lambda_2$ and $\lambda_3$.)
Some subtleties appear here.
First, the sequence $\left(  q_{n,k}\right)_{n \geq 0}$ thus completed can now take negative values. This will force us to go back to the non-completed probability sequence $\left(  q_{n,k}\right)_{n \geq n_0}$ to use the Quasi-Powers theorem. Secondly, 
if $\lambda_2 \neq \lambda_3$, the support is not necessarily finite any more.  It will add some difficulty to prove the analyticity in $u$ required by the Quasi-Powers theorem, which justifies the next item.

\textbf{Step 2: proving the analyticity of the coefficients.}
We show here that $q_n(u)$, defined as $\sum_{k \geq 0}q_{n,k}u^k$, is analytic at $u=1$ for every $n \leq n_0+1$.  It holds 
for $n = n_0$ and $n= n_0+1$ because by assumption, $q_n(u)$ is a polynomial (the support of $\left(q_{n,k}\right)_k$ is finite).  For the  numbers smaller than $n_0$, we show by induction on $\ell \in \{0,\dots,n_0\}$ that the quantity
$r_\ell(u)$, defined as $(u^{\lambda_2} + u^{\lambda_3})^\ell q_{n_0-\ell}(u)$, is a polynomial. Indeed, we observe that \eqref{eq:tildepnk} can be written as
\[\left(u^{\lambda_2}+u^{\lambda_3}\right)q_{n-2}(u) = 2n\,q_n(u) + 2\,(1-n)\,q_{n-1}(u),\]
which implies for every $\ell\geq 0$ by multiplying both sides by $\left(u^{\lambda_2}+u^{\lambda_3}\right)^{\ell-1}$:
\[ r_\ell(u) =  2\,(n_0-\ell+2)\,\left(u^{\lambda_2}+u^{\lambda_3}\right) r_{\ell-2}(u) - 2\,(n_0-\ell +1)\, r_{\ell-1}(u).\]
The last equality shows that the induction hypothesis is preserved, so by induction (the base case $\ell = 0$ and $\ell = 1$ are obvious), the series $r_{\ell}(u)$ is polynomial for  every $\ell \in \{0,\dots,n_0\}$. In particular, it means that $q_n(u)$ is analytic at $1$ for every $0 \leq n \leq n_0$.

\textbf{Step 3: solving the differential equation.}
Let us consider $Q(z,u) := \sum_{n,k\geq 0} q_{n,k} z^n\,u^k$ the \textit{completed}  generating function of the numbers $q_{n,k}$. In the same spirit as the proof of Proposition~\ref{prop:GFrec}, Equation \eqref{eq:tildepnk} can be translated in terms of a differential equation on $Q(z,u)$:
\[(1-z)\,\pd Q z(z,u) - \frac z 2 \,\left(u^{\lambda_2}+u^{\lambda_3} \right) \,Q(z,u) = q_1(u),\] 
where $q_1(u)$ is the coefficient of $z^1$ in $Q(z,u)$. This equation has for solution
\begin{equation}
Q(z,u) = q_0(u) e^{-\alpha(u)z}\,(1-z)^{-\alpha(u)} + q_1(u) e^{-\alpha(u)z}\,(1-z)^{-\alpha(u)} \, \int_0^z e^{\alpha(u)x}\,(1-x)^{\alpha(u)-1} dx,  
\label{eq:solu}
\end{equation}
where $\alpha(u) := \left(u^{\lambda_2}+u^{\lambda_3} \right)/2$ and $q_0(u)$ is the constant coefficient in $z$ of $Q(z,u)$.

\textbf{Step 4: finding a representation of $\boldsymbol{Q(z,u)}$.}
Let us prove that $Q(z,u)$ has a representation of the form
\[Q(z,u) = A(z,u) + B(z,u)\,(1-z)^{-\alpha(u)}.\] For that, we observe by a simple calculation that an antiderivative for $e^{\alpha(u)z}\,(1-z)^{\alpha(u)-1}$ is given by
\[- e^{\alpha(u)} \, \sum_{k \geq 0}  \frac {(1-z)^{\alpha(u)+k}} {(\alpha(u)+k)k!} (-\alpha(u))^k.\]
Thus, if $H(z,u)$ denotes the series $\sum_{k \geq 0} z^k/((u+k)k!)$ (which is analytic for every $z$ and for $u=1$), then \eqref{eq:solu} can be put into the form
\begin{multline*}Q(z,u) = -q_1(u)\,e^{\alpha(u)(1-z)}\,H(\alpha(u)z-\alpha(u),\alpha(u)) \\ +\left(q_0(u)+ q_1(u)\,e^{\alpha(u)}\,H(-\alpha(u),\alpha(u)) \right) e^{-\alpha(u)z} (1-z)^{\alpha(u)}, \end{multline*}
which is exactly the wanted representation. 

\textbf{Step 5: application of the Quasi-Powers Theorem.} As announced in Step 1, we use the Quasi-Powers Theorem (Theorem~\ref{theo:qp}) not on $Q(z,u)$ (since it could have negative coefficients), but on the non-completed probability generating function $\sum_{n \geq n_0} q_{n,k} z^n u^k$. The latter function has a representation of the form
$A(z,u) + B(z,u)\,(1-z)^{-\alpha(u)}$
since it differs from $Q(z,u)$ by an analytic function which is here $\sum_{n = 0}^{n_0-1}\,q_n(u) z^n$. Moreover, it is analytic at $(0,0)$ and it has non-negative coefficients. The variability condition is also satisfied since $\alpha'(1)+\alpha''(1) = \lambda_2^2/2 + \lambda_3^2/2$. The Quasi-Power Theorem thus proves that $Y_n$ converges to a Gaussian limit law with the announced properties.
\end{proof}

The next lemma, which is quite technical, shows how the error in $O(n^{-2})$ from~\eqref{eq:pnk} is propagated over the differences $p_{n,k} - q_{n,k}$, when $n$ goes to $+\infty$.
 
\begin{lem} There exists a family of constants $M_{n,k,\ell}$ with $n_0+2 \leq \ell \leq n$ and a constant $M$ such that  
\begin{itemize}
\item for every $n \geq \ell \geq n_0 + 2$,
\[\sum_{k \geq 0} M_{n,k,\ell} \leq M;\]
\item for every $n \geq n_0$ and $k$ positive,
\[|p_{n,k} - q_{n,k}| \leq \max_{k \geq 0}\, \left|p_{n_0,k} - q_{n_0,k}\right| + \max_{k \geq 0} \, \left|p_{n_0+1,k} - q_{n_0+1,k}\right| +   \sum_{\ell=n_0+2}^n \frac {M_{n,k,\ell}} {\ell^2}; \]
\end{itemize}
where $p_{n,k}$ is defined by Lemma~\ref{lem:asympteq} and $q_{n,k}$ can be any sequence defined by Lemma~\ref{l:Yn}. (In other words, $M$ and the constants $M_{n,k,\ell}$ do not depend on the sequence $q_{n,k}$.)
\label{l:diff}
\end{lem} 
\begin{proof} The lemma is proved by an induction on $n$. The statement is obvious for $n = n_0$ and $n = n_0 +1$.

For $n > n_0 + 1$, the combination of \eqref{eq:pnk} and \eqref{eq:tildepnk} leads to the inequality
\begin{multline*} \left|p_{n,k} - q_{n,k}\right| \leq \left(1-n^{-1}\right)\left|p_{n-1,k} - q_{n-1,k}\right| +  \frac{n^{-1}} 2 \,\left|p_{n-2,k-\lambda_2} - q_{n-2,k-\lambda_2}\right| \\  + \frac{n^{-1}} 2 \,\left|p_{n-2,k-\lambda_3} - q_{n-2,k-\lambda_3}\right| + O(n^{-2}).
\end{multline*}
Referring to the proof of Lemma~\ref{lem:asympteq}, we see that the error $O(n^{-2})$ in \eqref{eq:pnk} corresponds to the probability
$\prob{X_n \notin \bigcup_{i=1,2,3} \mathcal C_i \textrm{ and }  f(X_n) = k}$. This number is also $\mu_{n,k} \times \prob{X_n \notin \bigcup_{i=1,2,3} C_i}$, where $\mu_{n,k}$ is the conditional probability
\[\mu_{n,k} := \prob{f(X_n) = k \left| X_n \notin  \bigcup_{i=1,2,3} C_i \right.}.\]
(We have $\sum_{k \geq 0} \mu_{n,k} = 1$.) We have already stated that $\prob{X_n \notin \bigcup_{i=1,2,3} C_i} = O(n^{-2})$, so there exists a constant $M$ such that $\prob{X_n \notin \bigcup_{i=1,2,3} C_i}$ is smaller than $M\,n^{-2}$ for every $n \geq 0$. Using that fact and the induction hypothesis, the previous inequality becomes
\begin{multline*} \left|p_{n,k} - q_{n,k}\right| \leq \left(1-n^{-1}\right)\left(x + \sum_{\ell=n_0+2}^{n-1} \frac {M_{n-1,k,\ell}} {\ell^2}\right) +  \frac{n^{-1}} 2 \,\left(x + \sum_{\ell=n_0+2}^{n-2} \frac {M_{n-2,k-\lambda_2,\ell}} {\ell^2}\right) \\   + \frac{n^{-1}} 2 \,\left(x + \sum_{\ell=n_0+2}^{n-2} \frac {M_{n-2,k-\lambda_3,\ell}} {\ell^2}\right) + M\,\mu_{n,k}\,n^{-2},
\end{multline*}
where $x := \max_{k \geq 0}\, \left|p_{n_0,k} - q_{n_0,k}\right| + \max_{k \geq 0} \, \left|p_{n_0+1,k} - q_{n_0+1,k}\right|$.
Reorganising the terms, we find that
\[\left|p_{n,k} - q_{n,k}\right| \leq x + \sum_{\ell = n_0+2}^n \,\frac{M_{n,k,\ell}}{\ell^{2}},
\]
where we have set
\[M_{n,k,n} = M \mu_{n,k}, \quad \quad M_{n,k,n-1} = (1-n^{-1}) \, M_{n-1,k,n-1},\]
and for $n_0 + 2 \leq \ell \leq n-2$,
\[M_{n,k,\ell} = (1-n^{-1}) \, M_{n-1,k,\ell} + \frac{n^{-1}} 2 \, M_{n-2,k-\lambda_2,\ell} + \frac{n^{-1}} 2\, M_{n-2,k-\lambda_3,\ell}.\]
We have $\sum_{k \geq 0} M_{n,k,n} = M$ since $\sum_{k \geq 0} \mu_{n,k} = 1$. As for $n_0 + 2 \leq \ell \leq n-1$, the induction hypothesis shows that 
\[\sum_{k \geq 0} M_{n,k,\ell} \leq (1-n^{-1}) \, M + \frac{n^{-1}} 2 \, M + \frac{n^{-1}} 2\, M = M.\]
(The change of variable $k \leftarrow k-\lambda_i$ implies $\sum_{k \geq 0} M_{n-2,k-\lambda_i,\ell} = \sum_{k \geq 0} M_{n-2,k,\ell} \leq M$.) The induction is thus proved.
\end{proof}

%But $\left(1-n^{-1}\right)\,(n-1)^{-1/2} = \sqrt{1-n^{-1}}\,{n}^{-1/2} \leq (1-\frac{n^{-1}} 2)\,{n}^{-1/2}$, hence
%\[p_{n,k} \leq D\,{n}^{-1/2} - D\,{n}^{-1/2}\,\left(\frac{n^{-1}} 2 - n^{-1}\,(n-2)^{-1/2} n^{-1}\,(n-2)^{-1/2} \right).\]

% 
%\begin{lem} Let $f$ and $X_n$ be as defined in Theorem~\ref{theo:bigone} and $Y_n$ any sequence as defined in Lemma~\ref{l:Yn}. 
%For every $\mu >0, \sigma>0, t_1<t_2$ and $n \geq n_0$ we have
%\begin{multline*}\left|\prob{ t_1 \leq \dfrac{f(X_n) - \mu\,\ln(n)}{\sigma \sqrt{\ln(n)}}\leq t_2} - \prob{t_1 \leq\dfrac{Y_n - \mu\,\ln(n)}{\sigma \sqrt{\ln(n)}}\leq t_2}\right| \\
%\leq \left|\prob{t_1 \leq\dfrac{f(X_{n_0}) - \mu\,\ln({n_0})}{\sigma \sqrt{\ln({n_0})}}\leq t_2} - \prob{t_1 \leq\dfrac{Y_{n_0} - \mu\,\ln({n_0})}{\sigma \sqrt{\ln({n_0})}}\leq t_2}\right| + \frac 1 {\sqrt{n-n_0}}  
%\end{multline*}
%\end{lem}

We now have all the tools we need to show Theorem~\ref{theo:bigone}.

\begin{proof}[Proof of Theorem~\ref{theo:bigone}.]
 We want to prove that $\dfrac{f(X_n) - \mu\,\ln(n)}{\sigma \sqrt{\ln(n)}}$ converges in distribution to a standard Gaussian law, that is, for every $\varepsilon > 0$ and every real number $t$, there exists $n_1 \geq 0$ such that for every $n \geq n_1,$
\[ \left|\prob{ \dfrac{f(X_n) - \mu\,\ln(n)}{\sigma \sqrt{\ln(n)}}\leq t} -  F_{\mathcal N}(t) \right| \leq \varepsilon,\]
 where $F_{\mathcal N}(t)$ denotes the cumulative distribution function of the standard Gaussian law.
%
%Set $\alpha_n = \prob{ t_1 \leq \dfrac{f(X_n) - \mu\,\ln(n)}{\sigma \sqrt{\ln(n)}}  \leq t_2}$ and $\beta_n = \prob{t_1 \leq\dfrac{Y_n - \mu\,\ln(n)}{\sigma \sqrt{\ln(n)}}  \leq t_2}  $. Combining the fact that 
%\[\alpha_n = \sum_{\mu\ln(n) + t_1 \sigma \sqrt{\ln(n)} \leq k  \leq \mu\ln(n) + t_2 \sigma \sqrt{\ln(n)}}p_{n,k}\]
%and Lemma \ref{lem:asympteq}, we observe that
%\[\alpha_n = \left(1- n^{-1}\right) \, \alpha_{n-1} + 2\, n^{-1} \, %\alpha_{n-2} + O\left(\sqrt{\ln(n)}\,n^{-2}\right).\]

\textbf{1. Definition of $\boldsymbol{n_0}$.} The series $\sum_{\ell \geq 1} \ell^{-2}$ is convergent, hence its remainder tends to $0$. So there exists a number $n_0$ such that for every $n \geq n_0+2$, 
\begin{equation}
\sum_{\ell = n_0+2}^n \frac {M} {\ell^2} \leq \sum_{\ell = n_0+2}^{+\infty}  \frac {M} {\ell^2} \leq \frac \varepsilon 2, \label{eq:firstcut}
\end{equation}
where $M$ is the constant defined by Lemma~\ref{l:diff}. 

\textbf{2. Definition of an adapted sequence $\boldsymbol{(q_{n,k})}$.}
Let us define a sequence $\left(q_{n,k}\right)$ satisfying \eqref{eq:tildepnk} with initial conditions $q_{n_0,k} := p_{n_0,k}$ and $q_{n_0+1,k} := p_{n_0+1,k}$ for every $k \geq 0$. (The sequence $\left(q_{n,k}\right)$ satisfies Lemma~\ref{l:Yn}. The finiteness of the support comes from the fact there cannot be more integers $k$ such that $p_{n,k}\neq 0$ than the number of connected diagrams of size $n$.)

\textbf{3. Definition of $\boldsymbol{n_1}$ and first piece of the inequality.} We know by Lemma \ref{l:Yn} that the variable $\dfrac{Y_n - \mu\,\ln(n)}{\sigma \sqrt{\ln(n)}}$ converges in distribution to the standard Gaussian law. So there exists $n_1 \geq n_0 + 2$ such that for every $n \geq n_1$,
\begin{equation}
 \left|\prob{ \dfrac{Y_n - \mu\,\ln(n)}{\sigma \sqrt{\ln(n)}}\leq t} -   F_{\mathcal N}(t) \right| \leq \frac \varepsilon 2.
 \label{eq:scus}
 \end{equation}

\textbf{4. Second piece of the inequality.}
By definition, we have for every $n \geq n_1$,
\begin{multline*}
\left|\prob{ \dfrac{f(X_n) - \mu\,\ln(n)}{\sigma \sqrt{\ln(n)}}\leq t} - \prob{\dfrac{Y_n - \mu\,\ln(n)}{\sigma \sqrt{\ln(n)}}\leq t}\right| \\
= \left|\sum_{0 \leq  \, k \,  \leq \mu\,\ln(n) +  t \sigma \sqrt{\ln(n)}  } (p_{n,k} - q_{n,k})\right| \leq  \sum_{ k \geq 0  } \left|p_{n,k} - q_{n,k}\right|.
\end{multline*}
Lemma~\ref{l:diff} yields an upper bound for the difference $p_{n,k} - q_{n,k}$, so the previous number is bounded by
\[ \sum_{k \geq 0} \left(\max_{k \geq 0}\, \left|p_{n_0,k} - q_{n_0,k}\right| + \max_{k \geq 0} \, \left|p_{n_0+1,k} - q_{n_0+1,k}\right| +   \sum_{\ell=n_0+2}^n \frac {M_{n,k,\ell}} {\ell^2}  \right).\]
However, by definition of $(q_{n,k})$, $\max_{k \geq 0}\, \left|p_{n_0,k} - q_{n_0,k}\right| = \max_{k \geq 0} \, \left|p_{n_0+1,k} - q_{n_0+1,k}\right| = 0.$ We can then swap the sum over $k$ and the sum over $\ell$, and use the condition $\sum_{k \geq 0} M_{n,k,\ell} \leq M$ from Lemma~\eqref{l:diff} to obtain
\begin{equation}
\left|\prob{ \dfrac{f(X_n) - \mu\,\ln(n)}{\sigma \sqrt{\ln(n)}}\leq t} - \prob{\dfrac{Y_n - \mu\,\ln(n)}{\sigma \sqrt{\ln(n)}}\leq t}\right|
% \leq \sum_{\ell = n_0+2}^n \frac {\sum_{k \geq 0} M_{n,k,\ell}} {\ell^2}
 \leq \sum_{\ell = n_0+2}^n \frac M {\ell^2} \leq \frac \varepsilon 2,
 \label{eq:scur}
\end{equation}
where the last inequality comes from \eqref{eq:firstcut}.

\textbf{5. Conclusion.} The conjunction of \eqref{eq:scus} and \eqref{eq:scur} shows via a triangle inequality that
\[\left|\prob{ \dfrac{f(X_n) - \mu\,\ln(n)}{\sigma \sqrt{\ln(n)}}\leq t} -  F_{\mathcal N}(t) \right| \leq \frac \varepsilon 2 + \frac \varepsilon 2 = \varepsilon\]
for every $n \geq n_1$, as we had to prove.
\end{proof}

\subsection{Position of the first terminal chord}

In this subsection, we are interested by the average position of the first terminal chord for the intersection order. This parameter is relevant since it appears in the sum \eqref{sol} characterizing the Green function solution of \eqref{DSE}.

As an introductory remark, note that the first terminal chord is \textit{always} the chord with the rightmost endpoint, as stated by the following proposition.

\begin{prop} For every connected diagram, the first terminal chord is the chord that contains the last point of the diagram.
\end{prop}
\begin{proof}We proceed by induction on the number of chords. The property obviously holds when there is only one chord. Assuming now that there are several chords in the diagram, we remove the root chord from the diagram, which creates one or several connected components. By definition of the intersection order, the chords in the topmost component, which we denote $C_1$, are smaller than the other ones. But because it is the topmost component, $C_1$ must also contain the chord with the rightmost endpoint. Therefore, by using the induction hypothesis, the latter chord is the first terminal chord of $C_1$, hence the first terminal chord of the whole original diagram.
\end{proof}

Now let us turn on $f_n$, the random variable that returns the position of the first terminal chord, under the uniform distribution on connected chord diagrams of size $n$.

Note that $f_n$ does not satisfy the hypotheses of Theorem~\ref{theo:bigone}. Indeed, we can observe that the position of the first terminal chord for every diagram in $\mathcal C_2$ is $2$, regardless of the position of the first terminal chord of $C_2$.

This remark can be checked experimentally; the observed limit law is not Gaussian. In fact, it seems that $f_n/n$ converges to a law with a density, as shown by Figure \ref{fig:limitlaw}. We think that this density is $(1-s)^{-1/2}/2$, with $s \in [0,1)$. To our knowledge, such a limit law has never been observed on a class of combinatorial objects. This should be the subject of future work.  

\begin{figure}[ht!]
\begin{center}
\includegraphics[width=0.5\textwidth]{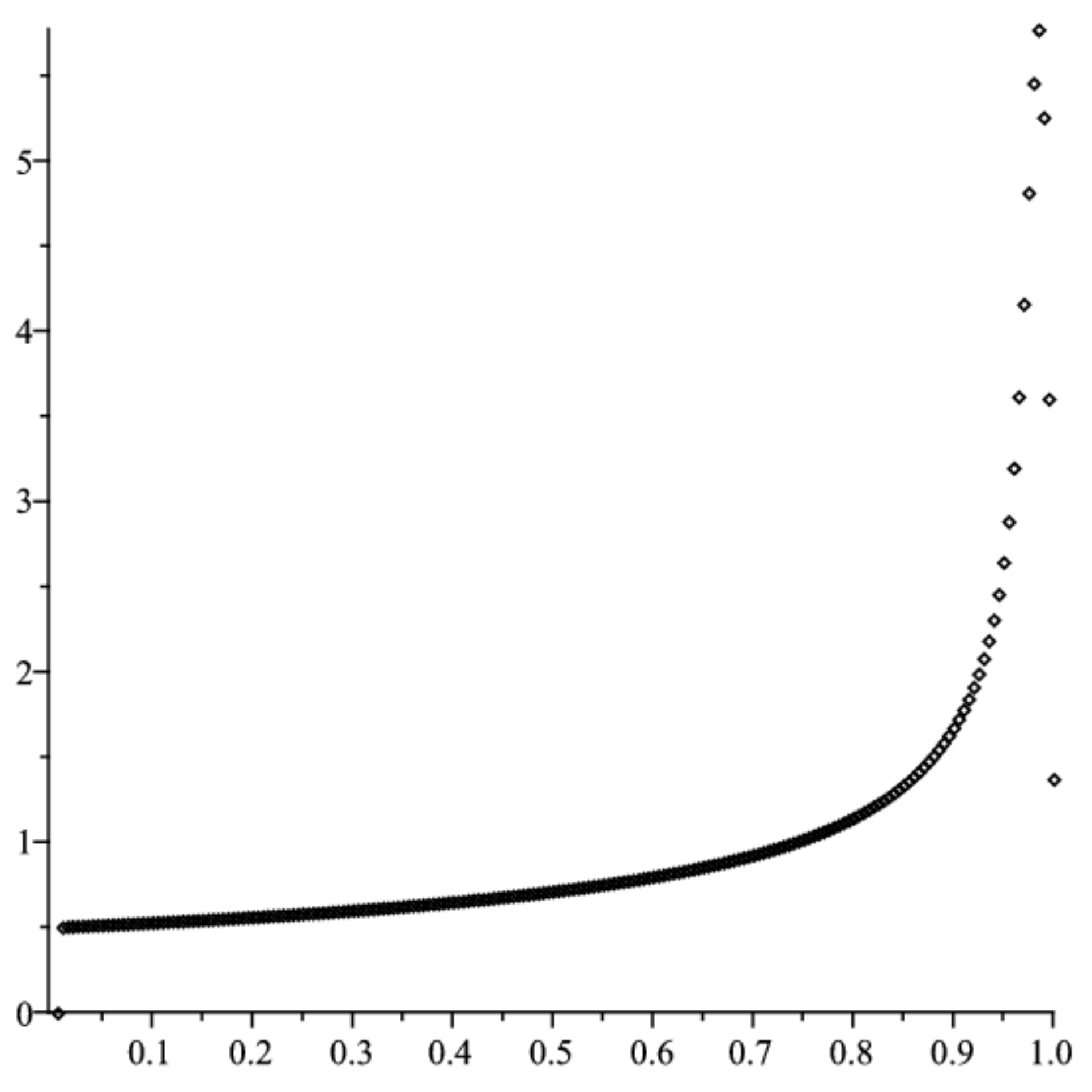}
\end{center}
\caption{Graph of $n \times \mathbb P(f_n/n=k)$ in terms of $k$, for $n=200$. If $f_n$ have a local limit law as we expect (cf \cite[p. 695]{flajolet-sedgewick} for a definition of \textit{local limit law}), then $n \times \mathbb P(x \leq f_n/n \leq x + dx)$ should converge to the density of the limit law of $f_n/n$ at the point $x$.}
\label{fig:limitlaw}
\end{figure}

We calculate here the expected value of this limit law.

\begin{theo} Let $f_n$ be the position of the first terminal chord of a uniformly distributed random connected chord diagram of size $n$. The expected value of $f_n$ is asymptotically equivalent to $ \dfrac{2 n} 3$.
\label{theo:evfn}
\end{theo}

Once again, the proof is based on the approximation of $f_n$ by another law which is easier to study. If we denote by $b(C)$ the position of the first terminal chord of a connected diagram $C$, and by $\mathcal C_1$, $\mathcal C_2$, $\mathcal C_3$ the three sets of connected diagrams defined in the beginning of this section, we can see that
\begin{itemize}
\item $b(C) = b(C_1) + 1$ for $C \in \mathcal C_1$;
\item  $b(C) = 2$ for $C \in \mathcal C_2$;
\item $b(C) = b(C_3) + 1$ for $C \in \mathcal C_3$.
\end{itemize}
A direct adaptation of the proof of Lemma \ref{lem:asympteq} shows then for $n \geq 2$, and $k \in \{3,\dots,n\}$
\begin{equation}
\prob{f_n = k} = \left(1 - \frac 1 n \right)\,\prob{f_{n-1} = k-1} + \frac 1 {2n}\,\prob{f_{n-2} = k-1} + O\left(n^{-2}\right),
\label{eq:fn}
\end{equation}
and \[\prob{f_n = 2} = \frac 1 {2n} +  O\left(n^{-2}\right).\]

We then define the numbers $g_{n,k}$ thanks to the recurrence
\begin{equation} g_{n,k} := \left(1 - \frac 1 n \right) g_{n-1,k-1} + \frac 1 {2n} g_{n-2,k-1}, \quad \quad g_{n,2} := \frac 1 {2n} 
\label{eq:propgn}
\end{equation}
for $n \geq 4$, and with initial conditions $g_{n,k} := \prob { f_n = k }$
for $n \in \ens{1,2,3}$. Remark by a straightforward induction that $\sum_{k \geq 0} g_{n,k} = 1$ for every integer $n$.

We start by proving that the expected values of $f_n$ and $g_{n,k}$ coincide asymptotically.

\begin{lem}
We have
\[\ev {f_n} - \sum_{k \geq 0}k\,g_{n,k} = o(n).\]
\label{lem:differon}
\end{lem}

\begin{proof}
Set $\varepsilon_n := \ev {f_n} - \sum_{k \geq 0}k\,g_{n,k}$. For $n \geq 2$, the law $f_n$ and the numbers $g_{n,k}$ have $\{2,\dots,n\}$ as a support, hence $\varepsilon_n = \sum_{k =2}^n k\,(\prob{f_n = k } - g_{n,k})$.

%Observe that for $j \in \{1,2\}$
%\[\sum_{k \geq 0} k (\prob{f_{n-j} = k-1} - g_{n-j,k-1}) = \sum_{k \geq 0} (k-1) (\prob{f_{n-j} = k-1} - g_{n-j,k-1}) + \sum_{k \geq 0} \prob{f_{n-j} = k-1} -  \sum_{k \geq 0} g_{n-j,k-1}  \]
Using Equations \eqref{eq:fn} and \eqref{eq:propgn}, we can then deduce that  
$\varepsilon_n = e_1 + e_2 + e_3$ for $n \geq 4$, where
\[e_1 := \left(1 - \frac 1 n\right)\,\sum_{k = 3}^n k \left( \prob {f_{n-1} = k - 1} - g_{n-1,k-1} \right)\]
\[e_2 := \frac 1 {2n}\, \sum_{k = 3}^n k \left( \prob {f_{n-2} = k - 1} - g_{n-2,k-1} \right) ,\]
\[e_3 := 2 (\prob{f_n = 2 } - g_{n,2}) + \sum_{k=2}^n O\left(n^{-2}\right) =   O\left(n^{-1}\right) ,\]
But we can note that
\begin{align*} \hspace*{-2pt}
e_1 \left(1 - \frac 1 n\right)^{-1} & = \sum_{k = 3}^n (k-1) \left( \prob {f_{n-1} = k - 1} - g_{n-1,k-1} \right) +  \sum_{k = 3}^n  \prob {f_{n-1} = k - 1} - \sum_{k = 3}^n  g_{n-1,k-1} \\
& = \varepsilon_{n-1} + 1 - 1 = \varepsilon_{n-1}.
\end{align*}
Similarly, $e_2 = \dfrac 1 {2n} \varepsilon_{n-2}$ so that
\[ \varepsilon_n = \left( 1 - \frac 1 n \right) \varepsilon_{n-1} + \frac 1 {2n} \, \varepsilon_{n-2} +  O\left(n^{-1}\right).\]

The sequence $\varepsilon_n/n$ is bounded (because $-n \leq f_n - g_n \leq n$), so has a limit point, let us say, $\ell$.
We have then
\[ \ell = \frac { (n-1)^2} {n^2} \ell + \frac {n-2} {2n^2} \ell +  O\left(n^{-2}\right).\]
The right-side member is asymptotically equivalent to $\ell - 3 \ell / (2n)$.  We must then have $\ell = 0$ so that this asymptotic estimate coincides with $\ell$. Consequently, the sequence $\varepsilon_n/n$ is bounded and have only one limit point, which is 0. Thus $\varepsilon_n/n$ tends to $0$, which means that $\ev {f_n} - \sum_{k \geq 0}k\,g_{n,k} = o(n)$.
\end{proof}

The next step is the explicit calculation of the generating function of the numbers $g_{n,k}$.

\begin{lem} The ordinary generating function $G(z,u)$ of the numbers $g_{n,k}$, namely
$\sum_{n,k \geq 0} g_{n,k}\,z^n\,u^k,$
is equal to
\begin{equation}
G(z,u) = e^{- \frac z 2} \left(1 - u\,z\right)^{- \frac 1 {2u}} \int_0^z  P(x,u) \, \frac  {e^{\frac x 2}} {1-x} \, \left(1 - u\,x\right)^{\frac 1 {2u} - 1} dx,
\label{eq:resG}
\end{equation}
where $P(x,u) = (1-x)(u+x\,u^2+ x^2\,u^2/4 +  x^2\,u^3/4) + x^3\,u^2/2$.
\label{lem:Gzu}
\end{lem}
\begin{proof}
Using \eqref{eq:propgn}, we can check that $G(z,u)$ satisfies the linear differential equation 
\[uzG(z,u) + 2 (1-uz) \pd G z (z,u) = \frac{u^2\,z^2}{1-z} + 2u+ 2u^2z+ \frac 1 2 u^2 z^2 - \frac 1 2 z^2 u^3.\]
We then verify that this differential equation is solved by \eqref{eq:resG}.
\end{proof}

\begin{proof}[Proof of Theorem \ref{theo:evfn}]
The sum $\sum_{k \geq 0} k g_{n,k}$ is the $n$th coefficient of the series $\pd G u (z,1)$, where $G$ is the generating function defined in Lemma \ref{lem:Gzu}. We are going to use the transfer theorem on $\pd G u (z,1)$. This series does not have a non-integral expression, but it is still possible to compute its singular expansion.

Write $G(z,u) = h_1(z,u) \int_0^z h_2(x,u) dx,$ 
where 
\[h_1(z,u) = e^{- \frac z 2} \left(1 - u\,z\right)^{- \frac 1 {2u}}, \quad \quad h_2(x,u) =  P(x,u) \, \frac  {e^{\frac x 2}} {1-x} \, \left(1 - u\,x\right)^{\frac 1 {2u} - 1}.\]
We have  \[\pd G u (z,1) = \pd {h_1} u (z,1) \int_0^z h_2(x,1) dx + h_1(z,1) \int_0^z \pd {h_2} u (x,1) dx.\] (Since we only have analytic functions, integration and differentiation with respect to $u$ are swappable.) One can explicitly compute the first part:
\[\pd {h_1} u (z,1) \int_0^z h_2(x,1) dx = \frac {z^2} 2 (1-z)^{-2} + \frac z 2 \ln(1-z) (1-z)^{-1}, \]
which is asymptotically equivalent to $(1-z)^{-2}/2$ when $z$ approaches $1$. Concerning the second part, we calculate $\pd {h_2} u (x,1)$ and observe that
\[\pd {h_2} u (x,1) \sim \frac {e^{\frac 1 2}} 4 (1-x)^{-5/2}.\] 
We then use Theorem VI.9 from \cite[p. 420]{flajolet-sedgewick} to integrate this expansion:
\[\int_0^z \pd {h_2} u (x,1) dx \sim \frac {e^{\frac 1 2}} 6 (1-z)^{-3/2},\]
and hence
\[h_1(z,1) \int_0^z \pd {h_2} u (x,1) dx \sim  \frac {1} 6 (1-z)^{-2}.\]

Finally we have $\pd G u (z,1) \sim \frac 2 3 (1-z)^{-2}$, so by the transfer theorem, we have  
\[ \sum_{k \geq 0} k g_{n,k} \sim \frac 2 3 n.\]
We conclude thanks to Lemma \ref{lem:differon}.
\end{proof}
%
%
%Theorem \ref{theo:evfn} results then from the combination of  and \ref{lem:evgn}.

\section{Conclusion}

In summary, this document establishes numerous exact and asymptotic results on connected chord diagrams. It shows how to compute the  next-to$\,^i$-leading log expansions, along with their asymptotic regimes. It also shows the Gaussian behaviour of many variables, like the number of terminal chords, and yields their means.

From a combinatorial point of view, this entire study is interesting on its own. It develops news methods to analyse parameters in a context which is not favourable to analytic combinatorics (\textit{a priori}). Moreover, it displays a non-Gaussian limit law, which seems to be new, and maybe deserves a deeper study.

Looking at this from a physical perspective, we observe the dominance of $f_0$ and $f_1$, which respectively denote the residue and the constant term of the Laurent expansion of the regularized Feynman integral of the one loop graph. This is particularly striking for the next-to$\,^i$-leading log expansions, whose asymptotic behaviour is governed by $f_0$ and $f_1$ (cf~\eqref{eq:f0f1}). But this dominance can also noted to a lesser extent to an unrestricted uniform distribution. In fact, by Corolla\-ries~\ref{cor:tn} and~\ref{cor:g1n}, the numbers $f_0$ and $f_1$ are on average exponentiated $n - \ln n$ and $\ln n/2$ times in the monomial $f_0^{|C|-k}\prod_{j = 2}^{k} f_{t_j-t_{j-1}}$, which leaves only $\ln n / 2$ extra factors for the other $f_i$ (always on average).

To have more information on these extra factors, it would be interesting to study the distribution of the gaps $t_j-t_{j-1}$ other than $1$. Conjecturally the number of $j$ such that $t_j-t_{j-1}=\ell$, where $\ell$ is fixed, asymptotically behaves like a Gaussian law with a mean and variance proportional to $\ln n/n ^{\ell-1}$. It should display a double regime: one is discrete -- a gap is equal to $1$ with a probability $1/2$; the other is continuous -- a gap conditioned to be different from $1$ should obey to a continuous limit law with mean $2\,n/\ln n$. The nature of the variance would be also interesting to know.

As for the next steps, the authors intend to generalize their results in the light of \cite{HYchord}. Specifically, the generalization should concern any number of primitives and Dyson-Schwinger equations of various shapes (including the QED shape).

%===============================
% BIBLIOGRAPHY
%===============================

\bibliographystyle{plain}
\bibliography{main}

\end{document}